\documentclass[11pt]{article}
\usepackage[english]{babel}
\usepackage[T1]{fontenc}
\usepackage{a4wide}
\usepackage[centertags]{amsmath}
\usepackage{amsfonts}
\usepackage{amssymb, verbatim}
\usepackage{amsthm, tabularx}
\usepackage{newlfont}
\usepackage{amssymb,amsbsy,times,fancyhdr,color}
\usepackage{amsmath}
\usepackage{graphicx}
\usepackage{multirow}
\usepackage{latexsym}
\usepackage{color, url}
\usepackage[normalem]{ulem}
\usepackage[title]{appendix}
\usepackage{chngcntr}

\usepackage{ulem}
\usepackage{hyperref}
\def\hi{\hat{\xi}}
\def\hh{\hat{h}}
\def\hhf{\hat{h}_f}

\def\hio{\hat \xi_0}

\def\hz{\hat{\zeta}}

\def\hx{\hat{x}}

\newcommand\ee{\mathbb{E}}
\newcommand\er{\mathbb{R}}
\newcommand\R{\mathbb{R}}

\newcommand\mfh{\mathfrak{h}}

\newcommand\ct{\mathcal{T}}

\newcommand\ind[1]{\mathbf{1}_{#1}}

\definecolor{jpred}{rgb}{0.99,0.1,0.1}

\newenvironment{Aproof}{\begin{proof}}{\end{proof}}

\usepackage{endnotes}
\let\footnote=\endnote

\usepackage{natbib}
 \bibpunct[, ]{(}{)}{,}{a}{}{,}%

\title{Energy imbalance market call options and the valuation of storage}
\author{John Moriarty\footnote{School of Mathematical Sciences, Queen Mary University of London, London E1 4NS, UK, email: {j.moriarty@qmul.ac.uk}}
\and
Jan Palczewski\footnote{School of Mathematics, University of Leeds, Leeds LS2 9JT, UK, email: {J.Palczewski@leeds.ac.uk}}}

\newcounter{licznik}[section]

\newtheorem{Definition}[licznik]{DEFINITION}

\newtheorem{Theorem}[licznik]{THEOREM}
\newtheorem{Lemma}[licznik]{LEMMA}
\newtheorem{Corollary}[licznik]{COROLLARY}
\newtheorem{Remark}[licznik]{REMARK}

\begin{document}
\maketitle

 \begin{abstract}
The use of energy storage to balance electric grids is increasing and, with it, the importance of operational optimisation from the twin viewpoints of cost and system stability. In this paper we assess the \emph{real option value} of balancing reserve provided by an energy-limited storage unit. The contractual arrangement is a series of American-style call options in an energy imbalance market (EIM), physically covered and delivered by the store, and purchased by the power system operator. We take the EIM price as a general regular one-dimensional diffusion and impose natural economic conditions on the option parameters.
In this framework we derive the \emph{operational strategy} of the storage operator by solving two timing problems: when to purchase energy to load the store (to provide physical cover for the option) and when to sell the option to the system operator. We give necessary and sufficient conditions for the finiteness and positivity of the value function -- the total discounted cash flows generated by operation of the storage unit. We also provide a straightforward procedure for the numerical evaluation of the optimal operational strategy (EIM prices at which power should be purchased) and the value function. This is illustrated with an operational and economic analysis using data from the German Amprion EIM.
\medskip


\end{abstract}

\section{Introduction}
\label{sec:intro}

The security of power systems is managed in real time by the System Operator (SO), who coordinates electricity supply and demand in a manner that avoids fluctuations in frequency or disruption of supply (see, for example, \cite{nztso}). In addition the SO carries out planning work to ensure that supply can meet demand, including the procurement of non-energy or {\em ancillary} services such as {\em operating reserve}, the capacity to make near real-time adjustments to supply and demand. Physically such adjustments may be provided by the control of thermal generation, demand or, increasingly, by the use of energy storage \citep{Xu2016, gridefr}. These resources have strongly differing operating characteristics: when compared to thermal generation, for example, energy storage is energy limited but can respond much more quickly. Storage also has important time linkages, since each discharge necessitates a corresponding recharge at a later time.

The financial procurement of operating reserve has an {\em option} character, as capacity is reserved in advance and randomly called for, potentially multiple times, in real time \citep{Just2008}. This is reflected in a two-price mechanism, with a reservation  payment plus an additional utilisation payment each time the reserve is called for. Since the incentivisation and efficient use of operating reserve for  system balancing is of increasing importance with growing penetration of variable renewable generation \citep{king2011flexibility}, several SOs have recently introduced real-time {\em energy imbalance markets} (EIMs) in which operating reserve is pooled, including in Germany \citep{ocker2015german} and California \citep{caiso2016, Lenhart2016}. Such markets typically involve the submission of bids and offers from several providers for reserves running across multiple time periods, which are then accepted, independently in each period, in price order until the real-time balancing requirement is met. As one provider can potentially be called upon over multiple consecutive periods, this reserve procurement mechanism is not well suited to energy-limited reserves such as energy storage. However, storage-oriented solutions are being pioneered in a number of markets including a recent tender by the National Grid in the UK \citep{gridefr} and various trials by state SOs in the US \citep{Xu2016}.

In this paper we consider the SO's planning problem of designing operating reserve contracts for energy limited storage devices such as batteries. In contrast to previous work on the pricing and hedging of energy options where settlement is financial (see for example \cite{benth2008stochastic} and references therein), we take account of the \emph{physical settlement} required in system balancing, considering also the limited energy and time linkages of storage. 
Physical feedback effects are investigated by studying the operational policy of the storage or \emph{battery operator} (which we abbreviate BO). We use {\em real options analysis} (RO) which is the application of option pricing techniques to the valuation of non-financial or ``real'' investments with flexibility \citep{borison2005real, dixit1994investment}. We consider the energy storage unit as the real asset, together with the operational flexibility of the BO, who observes the EIM price in real time. Since we take account of both the requirements of the SO and the operational policy of the BO, our work may also be interpreted as a form of principal-agent analysis.

A key question in RO analyses is the specification of the driving randomness \citep{borison2005real}, which in this paper is the probability law of the EIM price process under its physical measure. We thus model the EIM price 
to resemble statistically the observed historical dynamics \citep{Pflug2009, Ghaffari2013}. Unlike the prices of financial assets, energy imbalance prices (in common with electricity spot prices and commodity prices more generally) typically have significant mean reversion which should be modelled.
However even the simplest mean reverting models are not amenable to analytical treatment, due to the form of their infinitesimal generators and the presence of special functions in the Laplace transforms of their hitting times. In the present work we mitigate this problem by constructing optimal strategies only for {\em certain values} of the initial EIM price $X_0$, which are sufficient to solve the operational problems under study. By targeting a restricted, but nevertheless provably optimal, set of solutions in this way we are able to simplify the analysis compared to earlier work \citep{moriarty2014american} and hence to obtain results for {\em any} regular diffusion EIM price process $(X_t)_{t \geq 0}$, including mean reverting processes.

Options on balancing reserve have previously been proposed for the hedging of forward contracts by renewable power generators \citep{Ghaffari2013}. In the latter work the option exercise is of European type, that is, the exercise date is fixed. In contrast American style options, in which exercise is possible at any time (see for example \cite{DeWeert2011}), are required for applications in the continuous balancing of power systems. 
We consider a contract for a fixed quantity of balancing reserve, thus addressing the limited nature of energy storage. Importantly such a contract offers a potentially efficient solution to the issue of {\em physical cover}, since options can be issued only when the ancillary service is physically available. We assume that the SO sets the option parameters, namely the option premia (that is, the reservation and utilisation payments) plus an EIM price level $x^*$ at which the option is exercised. The premia are constant in our setup, reflecting the fact that balancing reserve is an ancillary service rather than a commodity. Our analysis thus focuses exclusively on the timing of the BO's actions. This dynamic modelling contrasts with previous economic studies of operating reserve in the literature, which have largely been static and concerned with prices and quantities \citep{Just2008}.
Further we restrict the SO's choice of option parameters in order to meet the following  {\em sustainability conditions}:
\begin{enumerate}
\item[{\bf S1.}] {\em The BO has a positive expected profit from the offer and exercise of the option.\label{goal:2}}
\item[{\bf S2.}] {\em The option cannot lead to a certain financial loss for the SO. \label{goal:1}}
\end{enumerate}
The present study addresses call options, or equivalently incremental capacity (defined as an increase in generation or equivalently a decrease in load). Put options, i.e., a decrease in generation or an increase in load, lead to a fundamentally different set of optimisation problems and are left for future research.

\subsection{Objectives}\label{sec:objectives}

Our main objective, dubbed the \emph{lifetime problem}, is to study the use of a dedicated battery to repeatedly provide balancing services through the considered contract to the SO.  To this end we first study the \emph{single option problem} in which the timing of a single energy purchase and option sale is optimised. We take into account the progressive degradation of the battery and consider the interests of both the SO and the BO, formulating the two mathematical aims {\bf M1} and {\bf M2} as follows.

We examine battery charging policies by identifying the {\em highest EIM price}, denoted $\check x$, at which the BO will buy energy:

\begin{enumerate}
\item[{\bf M1.}] {\em For the single and lifetime problems, find the highest EIM price $\check x$ at which the BO may buy energy when acting optimally.}
\end{enumerate}

As discussed above, for mathematical tractability our valuations will be restricted to certain initial prices $X_0$.  More precisely we have:

\begin{enumerate}
\item[{\bf M2.}] {\em For the single and lifetime problems, find the expected value of the total discounted cash flows (value function) for the BO corresponding to each initial EIM price $x \geq \check x$.}
\end{enumerate}

Finally we aim to provide a straightforward numerical procedure to explicitly calculate $\check x$ and the  value function (for $x \geq \check x$) in the lifetime problem.

\subsection{Methodology}\label{sec:setup}

The SO's system balancing challenge is real-time and continuous and we take the EIM price to be a continuous time stochastic process $(X_t)_{t \geq 0}$. Adapting the setup from \cite{moriarty2014american},
the following sequence of actions is considered:
\begin{enumerate}
\item[A1] The BO first selects a time to purchase a unit of energy on the EIM.
\item[A2] With this physical cover in place, the BO may then sell the call option to the SO in exchange for a premium $p_c \ge 0$.
\item[A3] The SO exercises the call option when the EIM price $X$ first lies above a given level $x^*$ and immediately receives one unit of energy in return for a utilisation payment $K_c \ge 0$. 
\end{enumerate}
The BO has timing flexibility in executing steps A1 and A2. For this reason we apply real options analysis from the BO's point of view in order to value the above sequence of actions.
Mathematically the problem is one of choosing two {\em optimal stopping times} corresponding to the two actions A1 and A2, based on the evolution of the stochastic process $X$. (The reader is refered to \citet[Chapter 1]{Peskir2006} for a thorough presentation of optimal stopping problems.) We centre our solution techniques around ideas of \cite{Beibel2000}, who characterise optimal stopping times using Laplace transforms of first hitting times for the process $X$ (see for example \citet[Section 1.10]{Borodin1996}). Methods and results from the single option analysis are then combined with a fixed point argument to find the optimal timings and  lifetime value function when the cycle A1--A3 is iterated indefinitely. 

Our methodological results feed into a growing body of research on timing problems in trading. In a financial context, \cite{Zervos2013} optimise the performance of ``buy low, sell high'' strategies, using the same Laplace transforms
to provide a candidate value function, which is later verified as a solution to quasi-variational inequalities. An analogous strategy in an electricity market using hydroelectric storage is studied in \cite{Carmona2010} where the authors use numerical methods to solve a related optimal switching problem. Our results differ from the above papers in two aspects. Our analysis is purely probabilistic, leading to simpler arguments that do not refer to the theory of PDEs and quasi-variational inequalities. Secondly, our characterisation of the value function and the optimal policy is explicit up to a one-dimensional non-linear optimisation which, as we demonstrate in an empirical experiment, can be performed in milliseconds using standard scientific software. 
Related to our lifetime analysis, \citet{carmona2008optimal} apply probabilistic techniques to study the optimal multiple-stopping problem for a general linear regular diffusion process and reward function. However the latter work deals with a finite number of option exercises in contrast to our lifetime analysis which addresses an infinite sequence of options via a fixed point argument. Our work thus yields results with a significantly simpler and more convenient structure.

The remainder of the paper is organised as follows. The mathematical formulation and some preliminary results are given in Section \ref{sec:results}. Our main results for the single option and lifetime problems are derived in Section \ref{sec:mainres}. In Section \ref{sec:examples} we show that, for several specific price processes $X$ which  incorporate mean reversion, solutions for {\em all} initial values $X_0$ can be obtained. An empirical illustration using real EIM data from the German Amprion SO is given in Section \ref{sec:empirics} and qualitative implications are drawn, while Section \ref{sec:conclusion} concludes.

\section{Formulation and preliminary results} \label{sec:results}
In this section we characterise the real option value 
of the sequence of actions A1--A3, and also the associated lifetime value of the store, using the theory of regular one-dimensional diffusions. Denoting by $(W_t)_{t \ge 0}$ a standard Brownian motion,  
let $X=(X_t)_{t \geq 0}$ be a (weak) solution of the stochastic differential equation:
\begin{equation}\label{eq:diffusion}
dX_t = \mu(X_t) dt + \sigma (X_t) dW_t,
\end{equation}
with boundaries $a \in \er \cup \{-\infty\}$ and $b \in \er \cup \{ \infty \}$. The solution of this equation with the initial condition $X_0 = x$ defines a probability measure $\mathbb{P}^x$ and the related expectation operator $\ee^x$. We assume that the boundaries are natural, i.e. the process cannot reach them in finite time, and that $X$ is a regular diffusion process, meaning that the state space $I:=(a,b)$ cannot be decomposed into smaller sets from which $X$ cannot exit. The existence and uniqueness of such an $X$ is guaranteed if the functions $\mu$ and $\sigma$ are Borel measurable in $I$ with $\sigma^2>0$, and 
\begin{equation}
\label{LI}
\text{$\forall \; y \in I,\ \exists \; \varepsilon>0$ such that } \int_{y-\varepsilon}^{y+\varepsilon}\frac{1 + |\mu(\xi)|}{\sigma^2(\xi)}\,d\xi < +\infty,
\end{equation}
(see \citet[Theorem 5.5.15]{karatzas1991}; condition \eqref{LI} holds if, for example, $\mu$ is locally bounded and $\sigma$ is locally bounded away from zero).
Necessary and sufficient conditions for the  boundaries $a$ and $b$ to be natural are formulated in Theorem 5.5.29 of the latter book. In particular, it is sufficient that the scale function 
\[
p(x) := \int_c^x \exp \left( -2 \int_c^z \frac{\mu(u)}{\sigma^2(u)} du \right) dz, \quad x \in I,
\]
converges to $-\infty$ when $x$ approaches $a$ and to $+\infty$ when $x$ approaches $b$. (Here  $c \in I$ is arbitrary and the condition stated above does not depend on its choice.) These conditions are mild, in the sense that they are satisfied by all common diffusion models for commodity prices, including those in Section \ref{sec:examples}. 

Denote by $\tau_x$ the first time that the process $X$ reaches $x \in I$, so that 
\begin{equation} \label{eqn:deftaux}
\tau_x = \inf \{ t \ge 0: X_t = x \}.
\end{equation}
For $r>0$, define
\begin{equation}\label{eq:egf}
\psi_r(x) = \begin{cases}
          \ee^x \{ e^{-r \tau_c}\}, & x \le c,\\
          1/\ee^c \{ e^{-r \tau_x}\}, & x > c,
          \end{cases}
\qquad
\phi_r(x) = \begin{cases}
          1/\ee^c \{ e^{-r \tau_x}\}, & x \le c,\\
          \ee^x \{ e^{-r \tau_c}\}, & x > c,
          \end{cases}
\end{equation}
for any fixed $c \in I$ (different choices of $c$ merely result in a scaling of the above functions). It can be verified directly that function $\phi_r(x)$ is strictly decreasing in $x$ while $\psi_r(x)$ is strictly increasing, and for $x, y \in I$ we have
\begin{equation}\label{eqn:discount_stop}
\ee^x \{ e^{-r \tau_y} \} = 
\begin{cases}
\psi_r (x) / \psi_r(y), & x < y,\\
\phi_r (x) / \phi_r(y), & x \ge y.
\end{cases}
\end{equation}
It follows, for example, from \citet[Section II.5]{Borodin1996} that $\psi_r$ and $\phi_r$ are $r$-excessive. (A nonnegative function $f$ is said to be {\em $r$-excessive} if $f(x) \geq \ee^x\{e^{-r \tau}f(X_\tau)\}$ for all stopping times $\tau$ and all $x \in I$.) Moreover, since the boundaries $a,b$ are natural, we have $\psi_r(a+) = \phi_r(b-) = 0$ and $\psi_r(b-) = \phi_r(a+) = \infty$ \citep[Section II.1]{Borodin1996}.

\subsection{Optimal stopping problems and solution technique}
\label{sec:osmethod}
The class of {optimal stopping problems} which we use in this paper is
\begin{equation}\label{eqn:full_stopping_problem}
v(x) = \sup_{\tau} \ee^x \{ e^{-r \tau} \vartheta(X_\tau) \ind{\tau < \infty} \},
\end{equation}
where the supremum is taken over the set of all (possibly infinite) stopping times. Here $\vartheta$ is the {\em payoff} function and $v$ is the {\em value} function. If a stopping time $\tau^*$ exists which achieves the equality \eqref{eqn:full_stopping_problem} we call this an {\em optimal} stopping time. Also, if $v$ and $\vartheta$ are continuous then 
the set
\begin{equation}\label{eqn:ost}
\Gamma := \{x \in I: v(x)=\vartheta(x)\} \qquad 
\end{equation}
 is a closed subset of $I$. Under general conditions \citep[Chapter 1]{Peskir2006}, which are satisfied by all stopping problems studied in this paper, $\tau^* = \inf\{ t \ge 0:\ X_t \in \Gamma\}$ is the smallest optimal stopping time and the set $\Gamma$ is then called the \emph{stopping set}.

Note that if $\sup_{x} \vartheta(x) \le 0$ then no choice of the stopping time $\tau$ gives a value function greater than $0$. The optimal stopping time in this case is given by $\tau = \infty$. In what follows we therefore assume 
\begin{equation}\label{eqn:sup_h_ge_0}
\sup_{x \in (a,b)} \vartheta(x) > 0.
\end{equation}

The following three lemmas provide an exhaustive list of possible types of solution to the stopping problem \eqref{eqn:full_stopping_problem}. Lemma \ref{lem:val_infinite} and \ref{lem:no_optimal} correspond to cases when there is no optimal stopping time but the optimal value can be reached in the limit by a sequence of stopping times.

\begin{Lemma}\label{lem:optim_taux}
Assume that there exists $\hat x \in I$ which maximises $\vartheta(x) / \phi_r (x)$ over $I$. 
Then the value function $v(x)$ is finite for all $x$, and for $x \ge \hat x$:
\begin{enumerate}
\item the stopping time $\tau_{\hat x}$ is optimal, 
\item
\(\displaystyle
v(x) = \frac{\vartheta(\hat x)}{\phi_r(\hat x)} \phi_r(x),
\)
\item any stopping time $\tau$ with $\mathbb{P}^x\big\{\vartheta(X_\tau) / \phi_r (X_\tau) < \vartheta(\hx) / \phi_r (\hx)\big\}>0$ is strictly suboptimal for the problem $v(x)$. 
\end{enumerate}
\end{Lemma}
\begin{Aproof}
Since $\phi_r$ is $r$-excessive, for any finite stopping time $\tau$
\[
\ee^x \{ e^{-r \tau} \phi_r(X_{\tau}) \} \le \phi_r(x).
\]
Let now $\tau$ be a stopping time taking possibly infinite values. Let $b_n$ be an increasing sequence converging to $b$ with $b_1 > x$, the initial point of the process $X$. Then $\tau_{b_n}$ is an increasing sequence of stopping times converging to infinity and 
\begin{align*}
\phi_r(x) &\ge \liminf_{n \to \infty} \ee^x \{ e^{-r (\tau \wedge \tau_{b_n}) } \phi_r(X_{\tau \wedge \tau_{b_n}}) \}\\
&\ge \ee^x \{ \liminf_{n \to \infty} e^{-r (\tau \wedge \tau_{b_n}) } \phi_r(X_{\tau \wedge \tau_{b_n}}) \} = \ee^x \{ e^{-r \tau} \phi_r(X_\tau) \ind{\tau<\infty} \},
\end{align*}
where $\phi_r(b-) = 0$ was used in the last equality.

For any stopping time $\tau$
\begin{equation}\label{eqn:upb}
\begin{aligned}
\ee^x \{ e^{-r \tau} \vartheta(X_\tau) \ind{\tau<\infty}\} 
&= \ee^x \left\{ e^{-r \tau} \phi_r(X_\tau) \frac{\vartheta(X_\tau)}{\phi_r(X_\tau)} \ind{\tau<\infty}\right\} \\
&\le \frac{\vartheta(\hat x)}{\phi_r(\hat x)} \ee^x \big\{ e^{-r \tau} \phi_r(X_\tau) \ind{\tau<\infty} \big\} \le \frac{\vartheta(\hat x)}{\phi_r(\hat x)} \phi_r(x), 
\end{aligned}
\end{equation}
where the final inequality follows from the first part of the proof and \eqref{eqn:sup_h_ge_0} (so $\frac{\vartheta(\hat x)}{\phi_r(\hat x)} > 0$). Hence, $v(x)$ is finite for all $x \in I$. To prove claim 1, note from \eqref{eqn:discount_stop} that for $x \geq \hat x$ the upper bound is attained by $\tau_{\hat x}$, which is therefore an optimal stopping time in the problem $v(x)$. The assumption on $\tau$ in claim 3 leads to strict inequality in \eqref{eqn:upb}, making $\tau$ strictly suboptimal in the problem $v(x)$.
\end{Aproof}

It is convenient to introduce the notation
\begin{equation}\label{eq:alim}
L:=\limsup_{x \to a} \frac{\vartheta(x)^+}{\phi_r(x)}.
\end{equation}
\begin{Lemma}\label{lem:val_infinite}
If the quantity $L$ in \eqref{eq:alim} is equal to positive infinity then the value function is infinite and there is no optimal stopping time.
\end{Lemma}

\begin{Aproof}
Fix any $x \in I$. Then for any $\hat x < x$ we have 
\[
\ee^{x} \{ e^{-r \tau_{\hat x}} \vartheta(X_{\tau_{\hat x}}) \} = \vartheta(\hat x) \frac{\phi_r(x)}{\phi_r(\hat x)},
\]
which converges to infinity for $\hat x$ tending to $a$ over an appropriate subsequence. Since the process is recurrent, the point $x$ can be reached from any other point in the state space with positive probability in a finite time. This proves that the value function is infinite for all $x \in I$.
\end{Aproof}

\begin{Lemma}\label{lem:no_optimal}
With the notation of \eqref{eq:alim}, assume $L < \infty$ and $L > \vartheta(x) / \phi_r(x)$ for all $x \in I$. Then there is no optimal stopping time and the value function equals $v(x) = L \phi_r(x)$.
\end{Lemma}
\begin{Aproof}
Recall that due to the supremum of $\frac{\vartheta}{\phi_r}$ being strictly positive we have $L > 0$. From the proof of Lemma \ref{lem:optim_taux}, for an arbitrary stopping time $\tau$ we have
\[
\ee^x \{ e^{-r \tau} \vartheta(X_\tau) \ind{\tau < \infty} \} = \ee^x \{ e^{-r \tau} \phi_r(X_\tau) \frac{\vartheta(X_\tau)}{\phi_r(X_\tau)} \ind{\tau < \infty} \}  
< L\, \ee^x \{ e^{-r \tau} \phi_r(X_\tau) \ind{\tau < \infty}\}  
\le L \phi_r(x).
\]
However, one can construct a sequence of stopping times that achieves this value in the limit. Take $x_n$ such that $\lim_{n \to \infty} \vartheta(x_n) / \phi_r(x_n) = L$ and define $\tau_n = \tau_{x_n}$. Then 
\[
\lim_{n \to \infty} \ee^x \{ e^{-r \tau_n} \vartheta(X_{\tau_n}) \} = \lim_{n \to \infty} \vartheta(x_n) \frac{\phi_r(x)}{\phi_r(x_n)} = \phi_r(x) L,
\]
so $v(x) = \phi_r(x) L$. This together with the strict inequality above proves that an optimal stopping time does not exist.
\end{Aproof}

The results developed in this section also have a `mirror' counterpart involving 
\begin{equation}\label{eq:R}
R:=\limsup_{x \to b} \frac{\vartheta(x)^+}{\psi_r(x)}
\end{equation}
rather than $L$. In particular, the value function is infinite if $R = \infty$, and
\begin{Corollary}\label{cor:taux}
If $\hat x \in I$ maximises $\vartheta(x) / \psi_r (x)$ then for any $x \le \hat x$ an optimal stopping time in the problem $v(x)$ is given by $\tau_{\hat x}$.
\end{Corollary}
This also motivates the assumptions of the following lemma which collects results from \citet[Section 5.2]{Dayanik2003}. 
\begin{Lemma} \label{lem:continuity}
Assume that $L, R < \infty$ and $\vartheta$ is locally bounded. Then the value function $v$ is finite and continuous on $(a,b)$.
\end{Lemma}
In the reminder of the paper, all the stopping problems considered will have a finite right-hand limit $R < \infty$. Therefore, whenever $L<\infty$, their value functions will be continuous.

\subsection{Single option problem formulation}
\label{sec:rewards}
Let $(X_t)_{t \geq 0}$ denote the EIM price. We will develop a mathematical representation of actions A1--A3 (see Section \ref{sec:setup}) for the single option contract. Starting from A3, the time of exercise by the system operator is the first time that the EIM price exceeds a predetermined level $x^*$:
\[
\hat \tau_e = \inf \{ t \ge 0: X_t \ge x^* \}.
\]
Given the present level $x$ of the EIM price, the expected net present value of the utilisation payment exchanged at time $\hat \tau_e$ can be expressed as follows thanks to \eqref{eqn:discount_stop}:
\begin{equation}
h_c (x) = E^x \{e^{-r\hat \tau_e} K_c \}=
\begin{cases}
K_c, & x \ge x^*,\\[5pt]
K_c \frac{\psi(x)}{\psi(x^*)}, & x < x^*.
\end{cases}\label{eq:one}
\end{equation}
Therefore, the optimal timing of action A2 corresponds to solving the following optimal stopping problem:
\[
\sup_{\tau} \ee^x \{ e^{-r\tau} \big(p_c + h_c(X_\tau)\big) \ind{\tau < \infty}\}.
\]
Since the utilisation payment $K_c$ obtained when the EIM price exceeds $x^*$ is positive and constant, as is the premium $p_c$, it is best to obtain these cashflows as soon as possible. The solution of the above stopping problem is therefore trivial: the contract should be sold immediately after completing action A1, i.e. immediately after providing physical cover for the option. Optimally timing the simultaneous actions A1 and A2, the \textbf{purchase of energy and sale of the option contract}, is therefore the core optimisation task. It corresponds to solving the following optimal stopping problem, whose payoff is non smooth:
\begin{equation}\label{eqn:single_option_stopping}
V_c(x) =\sup_{\tau} E^x \{ e^{-r \tau} \big(-X_\tau + p_c + h_c(X_\tau) \big) \ind{\tau < \infty}\} = \sup_{\tau} E^x \{ e^{-r \tau} h(X_\tau) \ind{\tau < \infty}\},
\end{equation}
where
\begin{equation}\label{eq:h1}
h(x) = - x + p_c + h_c(x).
\end{equation}
The function $V_c(x)$ is the real option value of the single option contract under our model.

\subsection{Lifetime problem formulation and notation}\label{sec:roundtrip}

In addition to having a design life of multiple decades, thermal power stations have the primary purpose of generating energy rather than providing ancillary services. In contrast electricity storage technologies such as batteries have a design life of years and may be dedicated to providing ancillary services. In this paper we take into account the potentially limited lifespan of electricity storage by modelling a multiplicative degradation of their storage capacity: each charge-discharge cycle reduces the capacity by a factor $A \in (0,1)$.

We now turn to considering the lifetime real option value of the store when used to sell an infinite sequence of single option contracts back-to-back. To this end, we suppose that a nonnegative {\em continuation value} $\zeta(x, \alpha)$ is also received at the same time as action A3. It is a function of the capacity of the store $\alpha \in (0,1)$ and the EIM price $x$, and represents the future proceeds from using the store to sell options back-to-back (either finitely or infinitely many times). 
 Since options are offered back-to-back, this continuation value enters the lifetime analysis as an additional payoff, which is received by the BO at the time of option exercise by the SO. 

The expected net present value of action A3 is now
\begin{equation}
h^\zeta (x, \alpha) := E^x \{e^{-r\hat \tau_e} \big(\alpha K_c + \zeta (X_{\hat \tau_e}, A \alpha)\big)\}
= \begin{cases}
         \big(\alpha  K_c + \zeta(x^*, A\alpha)\big)\frac{\psi(x)}{\psi(x^*)}, & x < x^*,\\
         \alpha  K_c + \zeta (x, A\alpha), & x \ge x^*,
         \end{cases}
\end{equation}
where $A \in (0,1)$ is the multiplicative decrease of storage capacity per cycle. Here the optimal timing of action A2 may be non trivial due to the continuation value $\zeta(x, \alpha)$. We will show however that for the functions $\zeta$ of interest in this paper, it is optimal to sell the option immediately after action A1, identically as in the single option case. The timing of action A1 requires the solution of the optimal stopping problem
\begin{equation}
\ct \zeta(x, \alpha) := \sup_{\tau} E^x \big\{e^{-r\tau} \big( - \alpha X_\tau + \alpha p_c + h^\zeta(X_\tau, \alpha)\big) \ind{\tau < \infty}\big\}.
\end{equation}
The \emph{optimal stopping operator} $\ct$ makes the dependence on $\zeta$ explicit: it maps $\zeta$ onto the real option value of a selling a single option followed by continuation according to $\zeta$. We define the \textbf{lifetime value function} $\hat V$ as the limit
\begin{equation}\label{eqn:hat_V}
\hat V(x) = \lim_{n \to \infty} \ct^n \textbf{0} (x, 1),
\end{equation}
(if the limit exists), where $\ct^n$ denotes the $n$-fold superposition of the operator $\ct$. Thus $\ct^n \textbf{0}$ is the real option value under our model of selling at most $n$ single options back-to-back. (Note that {\em a priori} it may not be optimal to sell all $n$ options in this case, since it is possible to offer fewer options and refrain from trading afterwards by choosing $\tau = \infty$.)

Calculation of the lifetime value function requires the analysis of a two-argument function. We will show now that this computation may be reduced to a function of the single argument $x$. Define $\zeta_0(x, \alpha) = 0$ and $\zeta_{n+1}(x, \alpha) = \ct \zeta_{n}(x, \alpha)$. We interpret $\zeta_n(x, \alpha)$ as the maximum expected wealth accumulated over at most $n$ cycles of the actions A1--A3 when the initial capacity of the store is $\alpha$.
\begin{Lemma}
We have $\zeta_n(x, \alpha) = \alpha \hat \zeta_n(x)$, where $\hat \zeta_n(x) = \zeta_n(x, 1)$. Moreover, $\hat \zeta_n(x) = \hat \ct^n \mathbf{0} (x)$, where
\begin{equation}\label{eqn:hatct}
\hat \ct \hat \zeta (x) = \sup_{\tau} \ee^x \big\{ e^{-r\tau} \big( - X_\tau +  p_c + \hat h^{\hat \zeta}(X_\tau) \big) \ind{\tau < \infty} \big\},
\end{equation}
and
\begin{equation}\label{eqn:hat_hc}
\hat h^{\hat \zeta}(x) = \begin{cases}
         \big(K_c + A \hat \zeta(x^*)\big)\frac{\psi(x)}{\psi(x^*)}, & x < x^*,\\
         K_c + A\hat \zeta (x), & x \ge x^*.
         \end{cases}
\end{equation}

\end{Lemma}
\begin{Aproof}
The proof is by induction. Clearly, the statement is true for $n=0$. Assume it is true for $n \ge 0$. Then
\[
\zeta_{n+1}(x, \alpha) = \ct \zeta_n(x, \alpha)
=
\alpha \sup_{\tau} \ee^x \big\{ e^{-r\tau} \big( -X_\tau +  p_c + \frac{1}{\alpha} h^{\zeta_n}(X_\tau, \alpha) \big)\ind{\tau < \infty} \big\},
\]
and
\[
\frac{1}{\alpha} h^{\zeta_n}(x, \alpha)
=
E^x \big\{e^{-r\hat \tau_e} \big(K_c + \frac{1}{\alpha} \zeta_n (X_{\hat \tau_e}, A\alpha)\big)\big\}\\
=
E^x\big\{e^{-r\hat \tau_e} \big(K_c + A\hat \zeta_n (X_{\hat \tau_e})\big)\big\}.
\]
Hence, $\zeta_{n+1}(x, \alpha) = \alpha \hat \ct \hat \zeta_n(x) = \alpha \zeta_{n+1}(x, 1)$. Consequently, $\hat \zeta_{n} = \hat \ct^n \mathbf{0}$.
\end{Aproof}

Assume that $\zeta_n(x, \alpha)$ converges to $\zeta(x, \alpha)$ as $n \to \infty$. Then, clearly, $\hat \zeta_n$ converges to $\hat \zeta(x) = \zeta(x, 1)$. It is also clear that $\zeta$ is a fixed point of $\ct$ if and only if $\hat \zeta$ is a fixed point of $\hat \ct$. Therefore, we have simplified the problem to that of finding a limit of $\hat \ct^n \mathbf{0} (x)$. The stopping problem $\hat \ct \hat \zeta$ will be called the \textit{normalised} stopping problem and its payoff denoted by
\begin{equation}\label{eq:hhat}
\hat h(x, \hz) =
\begin{cases}
-x + p_c + \frac{\psi_r(x)}{\psi_r(x^*)} \big(K_c + A\hat \zeta(x^*)\big), & x < x^*,\\
-x + p_c + K_c + A\hat \zeta(x), & x \ge x^*.
\end{cases}
\end{equation}
In particular, $\hat \ct \mathbf{0}$ coincides with the single option value function $V_c$.

{\em Notation.}
In the remainder of this paper a caret (hat) will be used over symbols relating to the normalised lifetime problem:
\begin{equation*}
\hat V(x) = \lim_{n \to \infty} \hat \ct^n \textbf{0} (x).
\end{equation*}

\subsection{Sustainability conditions revisited}\label{subsec:sustainability}
The sustainability conditions {\bf S1} and {\bf S2} introduced in Section \ref{sec:intro} are our standing economic assumptions for the model and options we consider. The next lemma expresses them quantitatively, making way for their use in the mathematical considerations below.
\begin{Lemma}\label{lem:equivsust}
When taken together, the sustainability conditions {\bf S1} and {\bf S2} are equal to the following quantitative conditions: 
\begin{enumerate}
\item[\textbf{S1*}:] $\sup_{x \in (a,b)} h(x) > 0$, and
\item[{\bf S2*}:] $p_c + K_c < x^*$.
\end{enumerate}
\end{Lemma}

\begin{Aproof}
If {\bf S1*} does not hold then the payoff from cycle A1--A3 is not profitable (on average) for any value of the EIM price $x$, so {\bf S1} does not hold. Conversely if {\bf S1*} holds then there exists $x$ such that $\hat\ct \mathbf{0}(x) \geq h(x) > 0$. For any other $x'$ consider the following strategy: wait until the process $X$ hits $x$ and proceed optimally thereafter. This results in a strictly positive expected value: $\hat\ct \mathbf{0}(x') > 0 $ and by the arbitrariness of $x'$ we have $\hat\ct \mathbf{0} > 0$. 

Suppose that {\bf S2*} holds. Then the SO makes a profit on the option (relative to simply purchasing a unit of energy at the exercise time $\hat \tau_e$, at the price $X(\hat \tau_e) \ge x^*$) in undiscounted cash terms. Considering discounting, the SO similarly makes a profit provided the EIM price reaches the level $x^*$ (or above) sufficiently quickly. Since this happens with positive probability for a regular diffusion, a certain financial loss for the SO is excluded. When {\bf S2*} does not hold, suppose first that $p_c + K_c > x^*$: then the SO makes a loss in undiscounted cash terms, and if the option is sold when $x \geq x^*$ then this loss is certain. In the boundary case $p_c + K_c = x^*$ the BO can only make a profit by purchasing energy and selling the option when $X_t < x^*$, in which case the SO makes a certain loss. This follows since instead of buying the option, the SO could invest $p_c > 0$ temporarily in a riskless bond, withdrawing it with interest when the EIM price rises to $x^*= p_c + K_c$. The loss in this case is equal in value to the interest payment.
\end{Aproof}

Notice that {\bf S1*} is always satisfied when $a \le 0$.

\section{Main results}\label{sec:mainres}

\subsection{Three exhaustive regimes in the single option problem}
\label{sec:single}

In this section we consider the {\em single option problem}. Recall that the sustainability assumptions, or equivalently assumptions {\bf S1*} and {\bf S2*}, are in force. Since the boundary $a$ is natural we have $\phi_r(a+)=\infty$. When $h$ is the single option payoff in \eqref{eq:h1}, the limit $L$ of \eqref{eq:alim} is then
\begin{equation}\label{eq:Lc}
L_c:=\limsup_{x \to a} \frac{-x}{\phi_r(x)},
\end{equation}
and we verify that $R<\infty$ in \eqref{eq:R} since by {\bf S2*}, $h$ is negative on $[x^*,\infty)$.

The general results obtained above are now specialised to the single option problem in the following theorem, which completes our aim {\bf M2} for the single option.
\begin{Theorem}(Single option problem)\label{thm:hammerA}
Assume that conditions {\bf S1*} and {\bf S2*} hold. With the definition \eqref{eq:Lc} 
there are three exclusive cases:
\begin{enumerate}
\item[(A)] $L_c \le \frac{h(x)}{\phi_r(x)}$ for some $x$  $\implies$ there is $\hat x < x^*$ that maximises $\frac{h(x)}{\phi_r(x)}$, and then, for $x \ge \hat x$, $\tau_{\hat x}$ is \textbf{optimal}, and
\begin{equation}\label{eqn:vfnA}
V_c(x) = \phi_r(x) \frac{h(\hat x)}{\phi_r(\hat x)}, \qquad x \ge \hat x.
\end{equation}
\item[(B)] $\infty > L_c > \frac{h(x)}{\phi_r(x)}$ for all x $\implies V_c(x) = L_c\, \phi_r(x)$ and there is \textbf{no optimal} stopping time.
\item[(C)] $L_c = \infty \implies V_c(x) = \infty$ and there is \textbf{no optimal} stopping time.
\end{enumerate}
Moreover, in cases A and B the value function $V_c$ is continuous.
\end{Theorem}
\begin{Aproof}
By condition {\bf S1*}, $h(y)$ is positive for some $y \in I$ and the value function $V_c (x) > 0$. For case A note first that the function $h$ is negative on $[x^*, b)$ by {\bf S2*}, see \eqref{eq:one} and \eqref{eq:h1}. Therefore, the supremum of $\frac{h}{\phi_r}$ is positive and must be attained at some (not necessarily unique) $\hx \in (a, x^*)$. The optimality of $\tau_{\hx}$ for $x \geq \hx$ then follows from Lemma \ref{lem:optim_taux}. Case B follows from Lemma \ref{lem:no_optimal} and the fact that $L_c > 0$. Lemma \ref{lem:val_infinite} proves case C. The continuity of $V_c$ follows from Lemma \ref{lem:continuity}.
\end{Aproof}

In case A, an optimal strategy is of a threshold type. When an arbitrary threshold strategy $\tau_{\tilde{x}}$ is used, the resulting expected value for $x \ge \tilde{x}$ is given by $\phi_r(x) h(\tilde x) / \phi_r(\tilde x)$. 
Figure \ref{fig:worthit} (whose problem data fall into case A) shows the potentially high sensitivity of the expected value of discounted cash flows for the single option with respect to the level of the threshold $\tilde x$. It is therefore important in general to identify the optimal threshold accurately. 
\begin{figure}[tb]
\begin{center}
\includegraphics[width=0.5\textwidth]{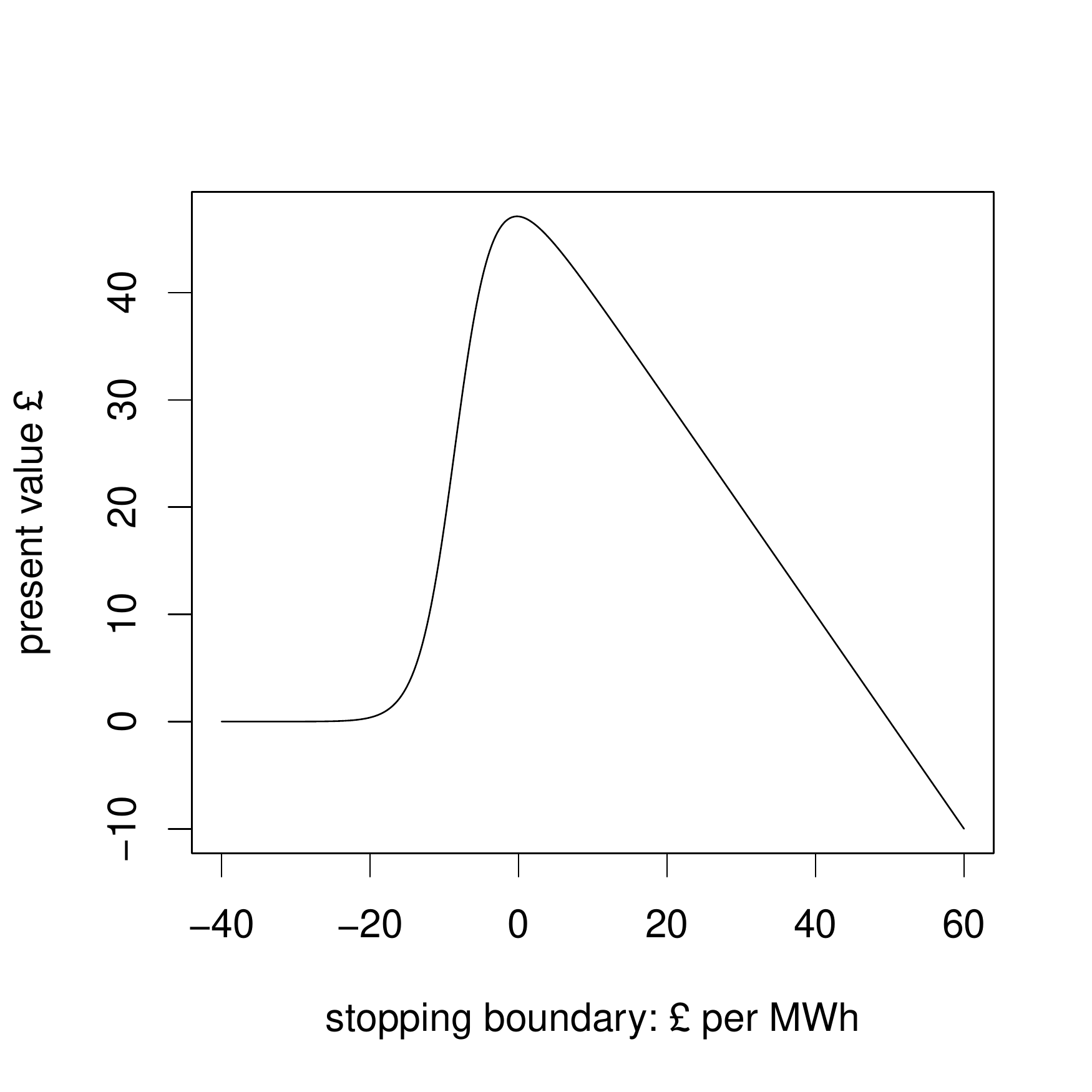}
\caption{Sensitivity of the expected value of the single option with respect to the stopping boundary. The EIM price is modelled as an Ornstein-Uhlenbeck process $dX_t = 3.42(47.66 - X_t )dt + 30.65 dW_t$ (time measured in days, fitted to Elexon Balancing Mechanism price half-hourly data from 07/2011 to 03/2014). The interest rate $r=0.03$, exercise level $x^* = 60$, the option premium $p_c = 10$, and the utilisation payment $K_c = 40$. The initial price is $X_0$ is set equal to $x^*$.
}
\label{fig:worthit}
\end{center}
\end{figure}

We now show that for commonly used diffusion price models, it is case A in the above theorem which is of principal interest. This is due to the mild sufficient conditions established in the following lemma which are satisfied, for example, by all of the examples in Section \ref{sec:examples}. Although condition 2(b) in Lemma \ref{cor:suff_case_A} is rather implicit, it may be interpreted as requiring that the process $X$ does not `escape relatively quickly to $-\infty$' (see  Appendix \ref{sec:lemrmk} for a further discussion and examples) and it is satisfied, for example, by the Ornstein-Uhlenbeck process of equation \eqref{eq:ousde}.

\begin{Lemma}\label{cor:suff_case_A}
If condition \textbf{S1$^*$} holds then:
\begin{enumerate}
\item The equality $L_c = 0$ implies case A of Theorem \ref{thm:hammerA}.
\item The following conditions are sufficient for $L_c = 0$:
\begin{enumerate}
\item $a>-\infty$,
\item $a=-\infty$ and $\lim_{x \to -\infty} \frac{x}{\phi_r(x)} = 0$.
\end{enumerate}
\end{enumerate}
\end{Lemma}
\begin{Aproof}
Condition \textbf{S1$^*$} ensures that $h$ takes positive values. Hence the ratio $\frac{h(x)}{\phi_r(x)} > 0 = L_c$ for some $x$. For assertion 2(a), recall from Section \ref{sec:results} that $\phi_r(a+) = \infty$ since the boundary $a$ is natural. Then we have $L_c = \limsup_{x \to a} (-x) / \phi_r (x) = 0$ as $a > -\infty$. In 2(b), the equality $L_c=0$ is immediate from the definition of $L_c$.  
\end{Aproof}

We now relate Theorem \ref{thm:hammerA} to our motivating problem of Section \ref{sec:setup}.
\begin{Corollary}\label{cor:1summ}
In the setting of Theorem \ref{thm:hammerA} for the single option problem, either 
\begin{itemize}
\item[(a)] the quantity 
\begin{equation}\label{eq:checkx}
\check x := \max \left\{  x \in I:\ \frac{h(x)}{\phi_r(x)} = \sup_{y \in I} \frac{h(y)}{\phi_r(y)} \right\}
\end{equation}
is well-defined, i.e., the set is non-empty. Then $\check x$ is the highest price at which the BO may buy energy when acting optimally (cf. objective \textbf{M1}), and we have $\check x < x^*$ (this is case A); or
\item[(b)] there is no price at which it is optimal for the BO to purchase energy. In this case the single option value function may either be infinite (case C) or finite (case B).  
\end{itemize}
\end{Corollary}
\begin{Aproof}
a) 
Since the maximiser $\hx$ in case A of Theorem \ref{thm:hammerA} is not necessarily unique, the set in \eqref{eq:checkx} may contain more than one point. Since $h$ and $\phi_r$ are continuous and all maximisers lie to the left of $x^*$, this set is closed and bounded from above, so $\check x$ is well-defined and  a maximiser in case A.
For any stopping time $\tau$ with $\mathbb{P}^x\{ X_\tau>\check x \} > 0$, it is immediate from assertion 3 of Lemma \ref{lem:optim_taux} that $\tau$ is not optimal for the problem $V_c(x)$, $x \ge \check x$.
Part b) follows directly from cases B and C of Theorem \ref{thm:hammerA}.
\end{Aproof}
In the single option problem, Corollary \ref{cor:1summ} achieves our first aim {\bf M1} from Section \ref{sec:objectives}. It confirms that it is optimal for the BO to buy energy only when the EIM price is strictly lower than the price $x^*$ which would trigger the exercise of the option by the SO. Thus the BO (when acting optimally) does not directly conflict with the SO's balancing actions.

\subsection{Two exhaustive regimes in the lifetime problem}\label{sec:lifetime}

Motivated by the real option valuation of an electricity storage unit, we now turn to the {\em lifetime problem} of valuing an infinite sequence of single option contracts. We begin by letting $\hz(x)$ in definition \eqref{eqn:hat_hc} be a general nonnegative continuation value depending only on the EIM price $x$, and studying the normalised stopping problem \eqref{eqn:hatct} in this case (the payoff $\hat h$ is therefore defined as in \eqref{eq:hhat}).

It follows from the optimal stopping theory reviewed in Section \ref{sec:osmethod} that our next definition,  of an {\em admissible} continuation function, is natural in our setup. In particular, the final condition corresponds to the assumption that the energy purchase occurs at a price below $x^*$.

\begin{Definition}(Admissible continuation value)
A continuation value function $\hz$ is {\em admissible} if it is continuous on $(a, x^*]$ and non-negative on $I$, with $\frac{\hat \zeta(x)}{\phi_r(x)}$ non-increasing on $[x^*, b)$.
\end{Definition}

\begin{Lemma}\label{lem:char}
Assume that conditions {\bf S1*} and {\bf S2*} hold. If $\hz$ is an admissible continuation value function then
\begin{equation}\label{eqn:lim_a_hh_h}
\limsup_{x \to a} \frac{\hat h(x, \hz)}{\phi_r(x)} = \limsup_{x \to a} \frac{-x}{\phi_r(x)} = L_c,
\end{equation}
and with cases A, B, C defined just as in Theorem \ref{thm:hammerA}:
\begin{enumerate}
\item In case A, there exists $\hat x \leq x^*$ which maximises $\frac{\hat h(x, \hz)}{\phi_r(x)}$ and $\tau_{\hx}$ is an optimal stopping time for $x\geq \hx$ with value function
\[
v(x) = \hat \ct \hz(x) = \phi_r(x) \frac{\hh(\hat x, \hz)}{\phi_r(\hat x)}, \qquad x \ge \hat x.
\]
Denoting by $\hx_0$ the corresponding $\hx$ in case A of Theorem \ref{thm:hammerA}, we have $\hx_0\le\hx$.
\item In case B, either 
\begin{itemize}
\item[a)] there exists $x_L \in (a,b)$ with $\frac{\hat h(x_L, \hz)}{\phi_r(x_L)} \geq  L_c$: then there exists $\hx \in (a,x^*]$ which maximises $\frac{\hat h(x, \hz)}{\phi_r(x)}$, and  $\tau_{\hx}$ is an optimal stopping time for $x\geq \hx$ with value function
$v(x) = \phi_r(x) \frac{\hh(\hat x, \hz)}{\phi_r(\hat x)}$ for $x \ge \hat x$; or
\item[b)] there does not exist $x_L \in (a,b)$ with $\frac{\hat h(x_L, \hz)}{\phi_r(x_L)} \ge L_c$: then the value function is $v(x) = L_c\, \phi_r(x)$ and there is no optimal stopping time.
\end{itemize}
\item In case C, the value function is infinite and there is no optimal stopping time.
\end{enumerate}
Moreover, the value function $v$ is continuous in cases A and B.
\end{Lemma}
\begin{Aproof}
Note that 
\begin{eqnarray}\label{eq:hmh}
h(x)= \hh(x, \mathbf 0) \le \hat h(x, \hz) =
\begin{cases}
h(x) + \frac{\psi_r(x)}{\psi_r(x^*)} A\hat \zeta(x^*), & x < x^*,\\
h(x) + A\hat \zeta(x), & x \ge x^*.
\end{cases}
\end{eqnarray}
This proves \eqref{eqn:lim_a_hh_h}, since $\lim_{x \to a} \psi_r(x) / \phi_r(x) = 0$. 
We verify from \eqref{eq:hmh} and the assumptions of the lemma that $R<\infty$ in \eqref{eq:R}. Hence, whenever $L_c < \infty$ the value function $v$ is finite and continuous by Lemma \ref{lem:continuity}.
As noted previously (in the proof of Theorem \ref{thm:hammerA}), $h$ is negative and decreasing on $[x^*,b)$, hence the ratio $h(x)/\phi_r(x)$ is strictly decreasing on that interval. It then follows from \eqref{eq:hmh} and  the admissibility of $\hz$ that the function $x \mapsto \frac{\hat h(x, \hz)}{\phi_r(x)}$ is strictly decreasing on $[x^*,b)$. Therefore the supremum of $x \mapsto \frac{\hat h(x, \hz)}{\phi_r(x)}$, which is positive by \eqref{eq:hmh} and \textbf{S1$^*$}, is attained on $(a, x^*]$ or asymptotically when $x \to a$. In cases 1 and 2a, the optimality of $\tau_{\hx}$ for $x \geq \hx$ then follows from Lemma \ref{lem:optim_taux}. To see that $\hx_0\le\hx$ in case 1, take $x < \hx_0$. Then from \eqref{eq:hmh} we have
\[
\frac{\hat h(x, \hz)}{\phi_r(x)} = \frac{h(x)}{\phi_r(x)} + \frac{\psi_r(x)}{\phi_r(x)} \frac{A\hat \zeta(x^*)}{\psi_r(x^*)} 
< \frac{h(\hx_0)}{\phi_r(\hx_0)} + \frac{\psi_r(\hx_0)}{\phi_r(\hx_0)} \frac{A\hat \zeta(x^*)}{\psi_r(x^*)} 
= \frac{\hat h(\hx_0, \hz)}{\phi_r(\hx_0)},
\]
since $x \mapsto \frac{\psi_r(x)}{\phi_r(x)}$ is strictly increasing. Case 2b follows from Lemma \ref{lem:no_optimal} and the fact that $L_c > 0$, while Lemma \ref{lem:val_infinite} proves case 3.
\end{Aproof}

Before proceeding we note the following technicalities.

\begin{Remark}\label{rem:const}
The value function $v$ in cases 1 and 2a of Lemma \ref{lem:char} satisfies the condition that $v(x) / \phi_r(x)$ is non-increasing on $[x^*, b)$. Indeed,
\[
\frac{v(x)}{\phi_r(x)} = \frac{\hh(\hat x, \hz)}{\phi_r(\hat x)} = const.
\]
for $x \ge \hx$.
\end{Remark}

\begin{Remark}
For case 3 of Lemma \ref{lem:char}, the assumption that $\frac{\hat \zeta(x)}{\phi_r(x)}$ is non-increasing on $[x^*, b)$ can be dropped.
\end{Remark}

We now wish to study the value of $n$ cycles A1--A3, and hence the lifetime value, by iterating the operator $\hat \ct$. To justify this approach, however, we must first check the timing of action A2 in the lifetime problem. With the actions A1--A3 defined as in Section \ref{sec:setup}, recall that the timing of action A2  is trivial in the single option case: after A1 it is optimal to perform A2 immediately. 

\begin{Lemma}
The timing of action A2 remains trivial when the cycle A1--A3 is iterated a finite number of times.
\end{Lemma}

\begin{Aproof}
Let us suppose that action A1 has just been carried out in preparation for selling the first option in a chain of $n$ options, and that the EIM price currently has the value $x$. Define $\tau_{A2}$ to be the time at which the BO carries out action A2. The remaining cashflows are (i) the first option premium $p_c$ (from action A2), (ii) the first utilisation payment $K_c$ (from A3), and (iii) all cashflows arising from the remaining cycles A1--A3 (there are $n-1$ cycles which remain available to the BO).  The cashflows (i) and (ii) are both positive and fixed, making it best to obtain them as soon as possible. The cashflows (iii) include positive and negative amounts, so their timing is not as simple. However it is sufficient to notice that 
\begin{itemize}
\item their expected net present value is given by an optimal stopping problem, namely, the timing of the {\em next} action A1:
\begin{equation}\label{eqn:cont}
 \sup_{\tau \geq \sigma^*} \ee^x \{ e^{-r \tau} h_{(iii)}(X_\tau) \ind{\tau < \infty} \},
\end{equation}
where $\sigma^* := \inf \{ t \ge \tau_{A2}: X_t \ge x^* \}$, for some suitable payoff function $h_{(iii)}$,
\item the choice $\tau_{A2} = 0$ minimises the exercise time $\sigma^*$ and thus maximises the value of component (iii), since the supremum in \eqref{eqn:cont} is then taken over the largest possible set of stopping times.
\end{itemize}
It is therefore best to set $\tau_{A2} = 0$, since this choice maximises the value of components (i), (ii) and (iii).  
\end{Aproof}

The next result addresses objective {\bf M2} for the lifetime problem by characterising, and establishing the existence of, the lifetime value function $\hat V$.

\begin{Lemma}\label{lem:perp_profit}
In cases A and B of Theorem \ref{thm:hammerA},
\begin{enumerate}
\item For each $n \geq 1$ the function  $\hz_n := \hat \ct^n\mathbf{0}$ satisfies the assumptions of Lemma \ref{lem:char} and is decreasing on $[x^*, b)$.
\item The functions $\hat \ct^n \mathbf{0}$ are strictly positive and uniformly bounded in $n$.
\item The limit $\hat\zeta = \lim_{n \to \infty} \hat \ct^n \mathbf{0}$ exists and is a strictly positive bounded function. Moreover, the lifetime value function $\hat V$ coincides with $\hz$.
\item The lifetime value function $\hat V$ is a fixed point of $\hat \ct$.
\end{enumerate}
\end{Lemma}
\begin{Aproof}
We prove part 1 by induction. The claim is clearly true for $n=1$. Assume it holds for $n$. Then Lemma \ref{lem:char} applies and $\hz_{n+1}(x) / \phi_r(x) = \hh(\hx, \hz_n) / \phi_r(\hx)$ for $x \ge \hx$ when the optimal stopping time exists and $\hz_{n+1}(x) / \phi_r(x) = L_c$ otherwise. Therefore, $\hz_{n+1}(x) = c \phi_r(x)$ for $x \ge x^*$ and some constant $c \ge 0$. Since $\phi_r$ is decreasing, we conclude that $\hz_{n+1}$ decreases on $[x^*, b)$. 

The monotonicity of $\hat\ct$ guarantees that if $\hat\ct \mathbf{0} > 0$ then $\hat\ct^n \mathbf{0} > 0$ for every $n$. For the upper bound, notice that 
\begin{align*}
\hat \ct \hz_n(x) &= \sup_{\tau} \ee^{x} \Big\{e^{-r \tau} \Big(p_c - X_\tau + \ee^{X_\tau} \big\{ e^{-r \hat \tau_e} \big(K_c + A\hz_n(X_{\hat \tau_e})\big) \big\} \Big)\ind{\tau < \infty}\Big\}\\
&\le \sup_{\tau} \ee^{x} \Big\{e^{-r \tau} \Big(p_c - X_\tau + K_c \ee^{X_\tau} \{ e^{-r \hat\tau_{e}} \} \Big)\ind{\tau < \infty}\Big\} + A \hz_n(x^*) = V_c(x) + A \hz_n(x^*),
\end{align*}
where $V_c = \hat\ct \mathbf{0}$  is the value function for the single option contract and the inequality follows from the fact that $\hz_n$ is decreasing on $[x^*, b)$. From the above we have $\hat \zeta_n (x) = \hat \ct^n \mathbf{0} (x) \le V_c(x) + \frac{1 - A^n}{1-A} V_c(x^*)$. Recalling that $A \in (0,1)$ yields that the $\hat \zeta_n(x)$ are bounded by $V_c(x) + \frac{1}{1-A} V_c(x^*)$, so there exists a finite monotone limit $\hat \zeta:=\lim_{n \to \infty} \hz_n$, and 
\begin{align*}
\hat \zeta (x)=\lim_{n \to \infty} \hat \ct \hz_n(x)
&= \sup_{n} \sup_{\tau} \ee^{x} \Big\{e^{-r \tau} \Big(p_c - X_\tau + \ee^{X_\tau} \big\{ e^{-r \hat \tau_e} \big(K_c + A\hz_n(X_{\hat \tau_e})\big) \big\} \Big)\ind{\tau < \infty}\Big\}\\
&= \sup_{\tau} \lim_{n \to \infty} \ee^{x} \Big\{e^{-r \tau} \Big(p_c - X_\tau + \ee^{X_\tau} \big\{ e^{-r \hat \tau_e} \big(K_c + A\hz_n(X_{\hat \tau_e})\big) \big\} \Big)\ind{\tau < \infty}\Big\}\\
&= \sup_{\tau} \ee^{x} \Big\{e^{-r \tau} \Big(p_c - X_\tau + \ee^{X_\tau} \big\{ e^{-r \hat \tau_e} \big(K_c + A\hz(X_{\hat \tau_e})\big) \big\} \Big)\ind{\tau < \infty}\Big\}\\
&= \hat \ct \hz(x),
\end{align*}
by monotone convergence. The equality of $\hat V$ and $\hz$ is clear from \eqref{eqn:hat_V}. 
 \end{Aproof}

We may now provide the following answer to objective \textbf{M1} for the lifetime problem.

\begin{Corollary}\label{cor:2summ}
In the setting of Lemma \ref{lem:char} with $\hz = \hat V$, either:
\begin{itemize}
\item[(a)] the quantity 
\begin{equation}\label{eq:checkx2}
\check x := \max \left\{  x \in I:\ \frac{\hh(x,\hz)}{\phi_r(x)} = \sup_{y \in I} \frac{\hh(y,\hz)}{\phi_r(y)} \right\}
\end{equation}
is well-defined, i.e., the set is non-empty. Then $\check x$ is the highest price at which the BO can buy energy when acting optimally in the lifetime problem, and we have $\check x < x^*$ (cases 1 and 2a); or
\item[(b)] there is no price at which it is optimal for the BO to purchase energy. In this case the lifetime value function may either be infinite (case 3) or finite (case 2b).
\end{itemize}
\end{Corollary}
\begin{Aproof}
The proof proceeds exactly as that of Corollary \ref{cor:1summ} with the exception of showing that $\check x < x^*$ (this is because Lemma \ref{lem:char} does not guarantee the strict inequality $\check x < x^*$). Assume then $\check x = x^*$. At the EIM price $X_t=\check x=x^*$ the call option exercise by the SO is immediately followed by the purchase of energy by the BO and this cycle can be repeated instantaneously, arbitrarily many times, when options are sold back-to-back. However since each such cycle is loss making for the BO by condition {\bf S2$^*$}, this strategy would lead to unbounded losses almost surely in the lifetime problem started at EIM price $x^*$ leading to $\hat V(x^*) = -\infty$. This would contradict the fact that  $\hat V>0$, so we conclude that $\check x <x^*$.
\end{Aproof}

Pursuing objective {\bf M2} a step further, we will show now that there are two regimes in the lifetime problem: either the lifetime value function is strictly greater than the single option value function, or it is only the purchase of power and the sale of the first option that contributes to the overall lifetime profit. In the latter case, the lifetime value equals the single option value.
\begin{Theorem}\label{thm:regimes}
There are two exclusive regimes:
\begin{itemize}
\item[($\alpha$)] $\hat V(x) > V_c(x)$ for all $x \ge x^*$,
\item[($\beta$)] $\hat V(x) = V_c(x)$  for all $x \ge x^*$ (or both are infinite for all $x$).
\end{itemize}
Moreover, in regime ($\alpha$) an optimal stopping time exists (that is, cases 1 or 2a of Lemma \ref{lem:char} hold) when the continuation value is $\hz = \hz_n = \hat\ct^n \mathbf{0}$ for $n > 0$ (that is, for a finite number of options), and when $\hz = \hat V$ (for the lifetime value function).
\end{Theorem}
\begin{Aproof}
We take the continuation value $\hz = V_c$ in Lemma \ref{lem:char} and consider separately its cases 1, 2a, 2b and 3. Firstly in case 3 we have $V_c = \infty$, implying that also $\hat V = \infty$ and we have regime $(\beta)$. 

Case 2 of Lemma \ref{lem:char} corresponds to case B of Theorem \ref{thm:hammerA}, when there is no optimal stopping time in the single option problem and $V_c(x)=L_c \phi_r(x)$ for all $x \in I$. Considering first case 2b and defining $\hz_n$ as in Lemma \ref{lem:perp_profit}, it follows that $\hz_2(x) = L_c \phi_r(x) = V_c(x)$ for $x \in I$ and consequently $\hat V = V_c$, which again corresponds to regime ($\beta$). 

In case 2a of Lemma \ref{lem:char}, suppose first that the maximiser $\hx \leq x^*$ is such that $\frac{\hat h(\hat x, \hz_1)}{\phi_r(\hat x)} = L_c$. Then for $x \ge x^* \ge \hat x$ we have $\hz_2(x) = \phi_r(x) \frac{\hat h(\hat x, \hz_1)}{\phi_r(\hat x)} = L_c \phi_r(x) $, which also yields regime ($\beta$). On the other hand, when $\frac{\hat h(\hat x, \hz_1)}{\phi_r(\hat x)} > L_c$ we have for $x \ge x^* \ge \hat x$ that $\hz_2(x) = \phi_r(x) \frac{\hat h(\hat x, \hz_1)}{\phi_r(\hat x)} > L_c \phi_r(x) = \hz_1(x)$, and so regime ($\alpha$) applies by the monotonicity of the operator $\hat\ct$. From the definition of $\hh$ in \eqref{eq:hhat}, and holding the point $\hat x \leq x^*$ constant, this monotonicity implies that $\frac{\hat h(\hat x, \hz_n)}{\phi_r(\hat x)} > L_c$ for all $n>1$ and that $\frac{\hat h(\hat x, \hat V)}{\phi_r(\hat x)} > L_c$. We conclude that case 2a of Lemma \ref{lem:char} applies (rather than case 2b) for a finite number of options and also in the lifetime problem.

Considering now the maximiser $\hx$ defined in case 1 of Lemma \ref{lem:char}, we have for $x \ge x^* \ge \hat x$ that
\[
\hz_2(x) = \phi_r(x) \frac{\hat h(\hat x, \hz_1)}{\phi_r(\hat x)} \ge \frac{\hat h(\hat x_0, \hz_1)}{\phi_r(\hat x_0)} > \frac{\hat h(\hat x_0, \mathbf{0})}{\phi_r(\hat x_0)} = \hz_1(x) = V_c(x),
\]
and regime $(\alpha)$ again follows by monotonicity. Also, trivially, case 1 of Lemma \ref{lem:char} applies for $\hz = \hz_n$ and $\hz = \hat V$.
\end{Aproof}

The following corollary follows immediately from the preceding proof.

\begin{Corollary}
Regime ($\beta$) holds if and only if  $\hat\ct^2 \mathbf{0}(x) = \hat\ct \mathbf{0}(x)$ for all  $x \ge x^*$. 
\end{Corollary}

To address the implicit nature of our answers to {\bf M1} and {\bf M2} for the lifetime problem, in the next section we provide results for the construction and verification of the lifetime value function and corresponding stopping time. For this purpose we close this section by summarising results obtained above (making use of additional results from Appendix \ref{sec:uniqueness}).

\begin{Theorem}\label{thm:lifetime_summary}
In the setting of Theorem \ref{thm:regimes} assume that regime ($\alpha$) holds. Then the lifetime value function $\hat V$ is continuous, is a fixed point of the operator $\hat \ct$ and $\hat \ct^n \mathbf{0}$ converges to $\hat V$ exponentially fast in the supremum norm. Moreover, there is $\check x < x^*$ such that $\tau_{\check x}$ is an optimal stopping time for $\hat \ct \hat V(x)$ when $x \ge \check x$ and, furthermore, $\check x$ is the highest price at which the BO can buy energy when acting optimally.
\end{Theorem}

\subsection{Construction and verification of the lifetime value function}\label{sec:lifetime_constr}

In this section we provide two lemmas for the lifetime problem which will be useful in its numerical solution. They are based on the problem's structure as summarised in Theorem \ref{thm:lifetime_summary}.
Firstly, Lemma \ref{lem:num1} provides a means of constructing the lifetime value function, together with the value $\check x$ of Theorem  \ref{thm:lifetime_summary}, using a one-dimensional search. Secondly, Lemma \ref{lem:neval} enables the result of such a search to be verified as the lifetime value function. We assume that regime ($\alpha$) of Theorem \ref{thm:regimes} holds.

\subsubsection{Construction}

\begin{Lemma}\label{lem:num1}
The lifetime value function evaluated at $x^*$ satisfies
\begin{equation}\label{eq:yy0}
\hat V(x^*) = \max_{z \in (a, x^*)} y(z),
\end{equation}
where
\begin{equation}\label{eq:yy}
y(z) := \frac{-z + p_c + \frac{\psi_r(z)}{\psi_r(x^*)} K_c}{\frac{\phi_r(z)}{\phi_r(x^*)} - \frac{\psi_r(z)}{\psi_r(x^*)}A}.
\end{equation}
\end{Lemma}
\begin{Aproof}
Fix $z \in (a, x^*)$. In the normalised lifetime problem of Section \ref{sec:roundtrip}, suppose that the strategy $\tau_z$ is used for each energy purchase. Writing $y$ for the total value of this strategy under $\mathbb{P}^{x^*}$, by  construction we have the recursion
\[
y=\frac{\phi_r(x^*)}{\phi_r(z)} \Big(-z + p_c + \frac{\psi_r(z)}{\psi_r(x^*)} \big(K_c + Ay\big)\Big).
\]
Rearranging, we obtain \eqref{eq:yy}. By Theorem  \ref{thm:lifetime_summary}, there exists an optimal strategy $\tau_{\check x}$ of the above form under $\mathbb{P}^{x^*}$ and \eqref{eq:yy0} follows.
\end{Aproof}

Hence under $\mathbb P^{x^*}$ an optimal stopping level $\hat x$ can be found by maximising $y(z)$ over $z \in (a, x^*)$. The value $\check x$ of Theorem  \ref{thm:lifetime_summary} is given by $\check x = \max\{ x:\ y(x) = \max_{z \in (a, x^*)} y (z) \}$.

\subsubsection{Verification}

We now provide a verification lemma which may be used to confirm the result if search \eqref{eq:yy0} is performed \emph{numerically},
or indeed to verify solutions found by  other means. The result is motivated by the following argument using Theorem \ref{thm:lifetime_summary}.

We claim that for all $x \in I$, $\hat\ct \hat V(x)$ depends on the value function $\hat V$ only through its value at $x=x^*$. The argument is as follows: when the BO acts optimally, the energy purchase occurs when the price is not greater than $x^*$: under $\mathbb P^x$ for $x \geq x^*$, this follows directly from Theorem \ref{thm:lifetime_summary}; under $\mathbb P^x$ for $x < x^*$, the energy is either purchased before the price reaches $x^*$ or one applies a standard dynamic programming argument for optimal stopping problems (see, for example, \cite{Peskir2006}) at $x^*$ to reduce this to the previous case. In our setup the continuation value is not received until the EIM price rises again to $x^*$ (it is received immediately if the energy purchase occurs at $x^*$). 

Suppose therefore that we can construct functions $V_i:I \to \mathbb R$, $i=1,2$, with the following properties:
\begin{enumerate}
\item[i)] $\hat \ct V_1 = V_2$,
\item[ii)] $V_1(x^*)=V_2(x^*)$,
\item[iii)] for $i=1,2$, the highest price at which the BO buys energy in the problem $\hat \ct V_i$ is not greater than $x^*$.
\end{enumerate}
Then we have $V_2 = \hat \ct V_1=\hat \ct  V_2$, so that $V_2$ is a fixed point of $\hat \ct$. 

We postulate the following form for $V_i$: given $y>0$ take
\begin{eqnarray}
V_1(x)&=&\hio^y(x) := \ind{x \le x^*} y,\\
V_2(x)&=&\hi^y(x):=\hat \ct \hio^y (x).
\end{eqnarray}
For convenience define $\mfh(x,y)$ to be the payoff in the lifetime problem when the the continuation value is $\hi^y_0$. Thus we have
\begin{eqnarray}
\mfh(x, y) &=& \hh(x, \hi^y_0), \label{eq:hiz}
\\
\hi^y(x) &=& \hat \ct \hio^y (x) = \sup_{\tau} \ee^x \big\{ e^{-r\tau} \mfh\big(X_\tau, y\big)\ind{\tau < \infty} \big\}. \label{eq:hi}
\end{eqnarray}
\begin{Lemma}\label{lem:neval}
Suppose that $\hat x \in (a, x^*)$ satisfies the system
\begin{eqnarray}
\frac{\mfh(\hat x, y)}{\phi_r(\hat x)} &=& \sup_{x \in (a, x^*)} \frac{\mfh(x, y)}{\phi_r(x)}, \label{eq:defy1}\\
y&=&\frac{\phi_r(x^*)}{\phi_r(\hat x)}\mfh(\hat x,y),\label{eq:defy}\\
y&>&0. \label{eq:yspos}
\end{eqnarray}
Then the function $\hi^y$ of \eqref{eq:hi}
is a fixed point of $\hat \ct$, is continuous and strictly positive, and
\begin{equation}\label{eqn:neval1}
\hi^y(x) = \frac{\phi_r(x)}{\phi_r(x^*)} y, \quad \text{for $x \ge \hat x$}.
\end{equation}
\end{Lemma}
\begin{Aproof}
Consider first the problem \eqref{eq:hi} with $x \ge \hat x$. By construction $\hi^y_0$ is an admissible continuation value in Lemma \ref{lem:char}, and  cases 1 or 2a must then hold due to the standing assumption for this section that regime ($\alpha$) of Theorem \ref{thm:regimes} is in force. By 
\eqref{eq:defy1} the stopping time $\tau_{\hat x}$ is optimal, and the problem's value function $\hi^y$ has the following three properties. Firstly, $\hi^y$ is continuous on $I$ by Lemma \ref{lem:continuity}. Secondly, using \eqref{eq:defy} we see that $\hi^y$ satisfies \eqref{eqn:neval1}. This implies thirdly that $\hi^y/\phi_r$ is constant on $[x^*,b)$ and establishes that $\hi^y(x^*)=y$, giving property ii) above. Since $y>0$ by \eqref{eq:yspos}, the strict positivity of $\hi^y$ everywhere follows as in part 1 of the proof of Lemma \ref{lem:equivsust}.
 Our standing assumption {\bf S2*} implies that the payoff $\mfh(x, y)$ of \eqref{eq:hiz} is negative for $x > x^*$, which establishes property iii) for problem \eqref{eq:hi}.
 
The three properties of $\xi^y$ established above make it an admissible continuation value in Lemma \ref{lem:char}, so we now consider the problem $\hat \ct \xi^y$ for $x \ge \hat x$. Under $\mathbb P^{x}$ for $x \geq x^*$, claim 2 of Lemma \ref{lem:optim_taux} prevents the BO from buying energy at prices greater than $x^*$ when acting optimally; under $\mathbb P^{x}$ for $x < x^*$, the dynamic programming principle mentioned above completes the argument.
\end{Aproof}

The following corollary completes the verification argument, and also establishes the uniqueness of the value $y$ in Lemma \ref{lem:neval}.

\begin{Corollary}\label{cor:fp}
Under the conditions of Lemma \ref{lem:neval}:
\begin{enumerate}
\item[i)] the function $\hi^y$ coincides with the lifetime value function: $\hat V = \hi^y$,
\item[ii)] there is at most one value $y$  for which the system equations \eqref{eq:defy1} and \eqref{eq:defy} has a solution $\hat x \in (a, x^*)$.
\end{enumerate}
\end{Corollary}
\begin{Aproof}
i) We will appeal to Lemma \ref{lem:geometric_convergence} by refining property iii) above for the problem $\hat \ct V_2=\hat \ct \hi^y$ (as was done in the proof of Corollary \ref{cor:2summ}). Suppose that the BO buys energy at the price $x^*$. Then since the function $\hi^y$ is a fixed point of $\hat \ct$ under our assumptions, we may consider $\hat \ct \hi^y (x^*) = -x^* + p_c + K_c + \hi^y(x^*)$ and then {\bf S2*} leads to $\hat \ct \hi^y(x^*) < \hi^y(x^*)$ which is a contradiction. Thus from Lemma \ref{lem:geometric_convergence}, $\hat\ct^n \mathbf{0}$ converges to $\hi^y$ as $n\to \infty$. As the limit of $\hat\ct^n \mathbf{0}$ is the lifetime value function we obtain $\hat V = \hi^y$.

ii) Assume the existence of two such values $y_1\neq y_2$. Then \eqref{eqn:neval1} gives $\hat V(x^*) = \hi^{y_1}(x^*) = y_1 \neq y_2 = \hi^{y_2}(x^*) = \hat V(x^*)$, a contradiction.
\end{Aproof}

We recall here that, on the other hand, the value $\hat x$ in Lemma \ref{lem:neval} may not be uniquely determined (cf. part (a) of Corollary \ref{cor:2summ}). In this case the largest $\hat x$ satisfying the assumptions of Lemma \ref{lem:neval} is the highest price $\check x$ at which the BO can buy energy optimally.

\section{Complete solutions for specific EIM price models}\label{sec:examples}

The general theory presented above provides optimal stopping times for initial EIM prices $x \geq \check x$, where $\check x$ is the highest price at which the BO can buy energy optimally. In this section, for specific models of the EIM price we derive optimal stopping times for \emph{all possible} initial EIM prices $x \in I$ when the sustainability conditions {\bf S1*} and {\bf S2*} hold. Note that condition {\bf S2*} is ensured by the explicit choice of parameters. Verification of condition {\bf S1$^*$} is straightforward by checking, for example,  if the left boundary $a$ of the interval $I$ satisfies $a < p_c + \limsup_{x \to a} \frac{\psi_r(x)}{\psi_r(x^*)} K_c$, i.e., that $\limsup_{x \to a} h(x) > 0$. In particular, {\bf S1$^*$} always holds if $a = -\infty$. 

Our approach is to take a variety of specific models for the EIM price and combine the above general results with the geometric method drawn from Proposition 5.12 of \cite{Dayanik2003} (and also used in \cite{moriarty2014american}). In particular we construct the least concave majorant $W$ of the obstacle $H:[0,\infty) \to \R$, where 
\begin{equation}\label{eq:H}
H(y) := 
\begin{cases}
\frac{\hh(F^{-1}(y),\hz)}{\phi_r(F^{-1}(y))}, & y > 0,\\
\limsup_{x \to a} \frac{\hh(x,\hz)}{\phi_r(x)} =L_c, & y = 0,
\end{cases}
\end{equation}
(the latter equality was given in \eqref{eqn:lim_a_hh_h}).
Here the function $F(x)=\psi_r(x)/\phi_r(x)$ is strictly increasing with $F(a+)=0$. Writing $\hat \Gamma$ for the set on which $W $ and $H$ coincide, under appropriate conditions the smallest optimal stopping time is given by the first hitting time of the set $\Gamma:=F^{-1}(\hat \Gamma)$ \citep[Propositions 5.13--5.14]{Dayanik2003}. 

Two of the EIM price models we take are based on the Ornstein-Uhlenbeck (OU) process. This is a continuous time stochastic process with dynamics
\begin{equation}
dX_t =  \theta (\mu - X_t) dt + \sigma dW_t, \label{eq:ousde}
\end{equation}
where $\theta, \sigma > 0$ and $\mu \in \er$. It has two natural boundaries, $a=-\infty$ and $b=\infty$. This process extends the scaled Brownian motion model by introducing a mean reverting drift term $\theta (\mu - X_t)dt$, which may be taken as a consequence of the SO's corrective balancing actions. Appendix \ref{sec:oufacts} collects some useful facts about the Ornstein-Uhlenbeck process. In particular, when constructing $W$ it is convenient to note that $H''\circ F$ has the same sign as $(\mathcal{L}-r)h$, where $\mathcal{L}$ is the infinitesimal generator of $(X_t)$ defined as in Appendix \ref{sec:oufacts}.

\subsection{OU price process}\label{subsec:OU}

Assume now that the EIM price follows the OU process \eqref{eq:ousde} so that $L_c=0$ (see equation \eqref{eqn:dx_phi_limit} in Appendix \ref{sec:oufacts}) and, by Lemma \ref{cor:suff_case_A}, case A of Theorem \ref{thm:hammerA} applies. We are able to deal with the single and lifetime problems simultaneously by setting $\hz$ equal to $0$ for the single option and equal to (the positive function) $\hat V$ in the lifetime problem.  The results of Section \ref{sec:mainres}  yield that in both problems, the right endpoint of the set $\hat \Gamma$ equals $F(\check x)$ for some $\check x < x^*$. Further, since $\psi_r$ is a solution to $(\mathcal{L} - r) v = 0$ and $\check x < x^*$ we have for $x \le \check x$
\begin{align}
(\mathcal{L}-r)\hat{h}(x, \hz) &= (\mathcal{L}-r)\Big(-x + p_c+ \frac{\psi_r(x)}{\psi_r(x^*)}\big(K_c + A\hat \zeta(x^*)\big)\Big) \notag\\
&= (\mathcal{L}-r)(-x + p_c) \notag\\
&= (r+ \theta) x - r p_c - \theta\mu. \notag
\end{align}
Therefore, the function $(\mathcal{L}-r)\hh(\cdot,\hz)$ is negative on $(-\infty,B_0)$ and positive on $(B_0,\infty)$, where $B_0 = \frac{r p_c + \theta\mu}{r+\theta}$. This implies that $H$ is strictly concave on $(0, F(B_0))$ and strictly convex on $(F(B_0), \infty)$. Since the concave majorant $W$ of $H$ cannot coincide with $H$ in any point of convexity, so necessarily $\check x < B_0$ and $H$ is concave on $(0,F(\check x))$. Hence we conclude that $W$ is equal to $H$ on the latter interval and so $\Gamma = (-\infty,\check x)$.

\subsection{Shifted exponential price processes}\label{sec:shiftexp}
In order to first recover and then generalise results from \cite{moriarty2014american}, we henceforth assume the following shifted exponential model for the price process:
\begin{eqnarray}
f(z) &:=& D + d e^{bz}, \label{eq:fz}\\
X_t&=&f(Z_t), \label{eq:stack}
\end{eqnarray}
where $Z$ is a regular one-dimensional diffusion with natural boundaries $a^Z$ and $b^Z$ (we will use the superscripts $X$ and $Z$ where necessary to emphasise the dependence on the stochastic process). The idea is that $Z$ models the physical system imbalance process while $f$ represents a {\em price stack} of bids and offers which is used to form the EIM price. In this case  the left boundary for $X$ is $a=f(a^Z) \ge D$ and, 
by Lemma \ref{cor:suff_case_A}, $L_c=0$ and case A of Theorem \ref{thm:hammerA} applies. Rather than working with the implicitly defined process $X$, however, we may work directly with the process $Z$ by setting:
\begin{eqnarray}
z^*&:=&f^{-1}(x^*), \label{eq:z1}\\
h_f(z)&:=&-f(z) + p_c + \begin{cases}
\frac{\psi_r^Z(z)}{\psi_r^Z(z^*)} K_c, & z < z^*,\\                        
 K_c, & z \ge z^*,
                        \end{cases}
\label{eq:hf}\\
\hhf(z,\hz)&:=& 
\begin{cases}
-f(z) + p_c + \frac{\psi_r^Z(z)}{\psi_r^Z(z^*)} \big(K_c + A\hat \zeta(z^*)\big), & z < z^*,\\
-f(z) + p_c + K_c + A\hat \zeta(z), & z \ge z^*,
\end{cases}\label{eq:hhf}
\end{eqnarray}
and modifying the definitions for $\ct, \hat \ct, V_c$ and $\hat V$ accordingly.
We then have
\begin{Theorem}\label{thm:hammerA'} Taking definitions \eqref{eq:fz} and \eqref{eq:z1}--\eqref{eq:hhf}, 
assume that conditions {\bf S1*} and {\bf S2*} hold. Then 
\[
L_c:=\limsup_{z \to a^Z} \frac{-f(z)}{\phi_r^Z(z)} 
=0.
\] 
Also:
\begin{itemize}
\item[i)] (Single option) There exists $\hat z < z^*$ that maximises $\frac{h_f(z)}{\phi_r^Z(z)}$, the stopping time $\tau_{\hat z}$ is optimal for $z \ge \hat z$, and
\[
V_c(z) = \phi_r^Z(z) \frac{h_f(\hat z)}{\phi_r^Z(\hat z)}, \qquad z \ge \hat z.
\]
\item[ii)] (Lifetime problem) The lifetime value function $\hat V$ is continuous and a fixed point of $\hat\ct$. There exists $\tilde z \in (\hat z, z^*)$ which maximises $\frac{\hat h(z, \hat V)}{\phi_r^Z(z)}$ and $\tau_{\tilde z}$ is an optimal stopping time for $z\geq \tilde z$ with
\[
\hat V(z) = \hat \ct \hat V(z) = \phi_r^Z(z) \frac{\hh(\tilde z,\hat V)}{\phi_r^Z(\tilde z)}, \qquad z \ge \tilde z.
\]
\end{itemize}
\end{Theorem}
\begin{Aproof}
The proof follows from the one-to-one correspondence between the process $X$ and the process $Z$,  and direct transfer from Theorems \ref{thm:hammerA} and \ref{thm:lifetime_summary}.
\end{Aproof}

In some cases, explicit necessary and/or sufficient conditions for \textbf{S1$^*$} may be given in terms of the problem parameters. Assume that $a^Z = -\infty$ as in the examples studied below. If $p_c > D$ and $K_c \geq 0$, this is sufficient for the condition \textbf{S1$^*$} to be satisfied as then $h_f(z) \geq -f(z) + p_c > 0$ for sufficiently small $z$. When $p_c = D$ and $K_c > 0$, it is sufficient to verify that $e^{bz} = o \big(\psi_r^Z(z)\big)$ as $z \to -\infty$ since then $h_f(z) = -de^{bz} + \psi_r^Z(z)\, K_c / \psi_r^Z(z^*)$ for $z < z^*$. On the other hand, our assumption that \textbf{S1$^*$} holds necessarily excludes parameter combinations with $p_c - D=K_c = 0$, since the option writer then cannot make any profit because $h_f(z) \le 0$ for all $z$.

In Section \ref{sec:bmi} we take $Z$ to be the standard Brownian motion and recover results from the single option case of \cite{moriarty2014american} (the lifetime problem is formulated differently in the latter reference, where degradation of the store is not modelled). In Section \ref{sec:oui} we generalise to the case when $Z$ is an OU process.

\subsubsection{Brownian motion imbalance process}\label{sec:bmi}

When the imbalance process $Z=W$, the Brownian motion, we have
\[
(\mathcal{L}-r)\hhf(z,\hz) = (\mathcal{L}-r)(-f(z) + p_c) = de^{bz}\left\{r - \frac12 b^2\right\} + r(D - p_c).
\]
We have several cases depending on the sign of $(D-p_c)$ and $(r - \frac12 b^2)$. 
\begin{enumerate}
\item Assume first that $r > \frac12 b^2$. 
\begin{enumerate}
\item[(i)] We may exclude the subcase $p_c \le D$, since then $H(y) = \frac{\hat{h}(z, \hz)}{\phi_r^Z(z)} |_{z = (F^Z)^{-1}(y)}$ is strictly convex on $(0, F^Z(z^*))$ for any $\hz$ and $\Gamma$ cannot intersect this interval, contradicting Theorem \ref{thm:hammerA'} and, consequently, violating \textbf{S1$^*$} or \textbf{S2$^*$}.
\item[(ii)] If $p_c > D$, $H$ is concave on $(0, F^Z(B))$ and convex on $(F^Z(B), \infty)$, where
\[
B = \frac{1}{b} \log \left(\frac{r (p_c - D)}{d (r - \frac12 b^2)} \right).
\]
By Theorem \ref{thm:hammerA'} and the positivity of $H$ on $(0, F^Z(\hat z))$ we have $\Gamma = (-\infty, \hat z]$ and $\Gamma = (-\infty, \tilde z]$ for the single and lifetime problems respectively, with $\tilde z < \hat z < B$.
\end{enumerate}
\item Suppose that $r < \frac12 b^2$.
\begin{enumerate}
\item[(i)] When $p_c \ge D$, the function $H$ is concave on $(0, \infty)$. Hence the stopping sets $\Gamma$ for single and lifetime problems have the same form as in case 1(ii) above.
\item[(ii)] If $p_c < D$, the function $H$ is convex on $(0, F^Z(B))$ and concave on $(F^Z(B), \infty)$. The set $\Gamma$ must then be an interval, respectively $[\hat z_0, \hat z]$ and $[\tilde z_0, \tilde z]$. For explicit expressions for the left and right endpoints for the single option problem, as well as sufficient conditions for \textbf{S1$^*$}, the reader is refered to \cite{moriarty2014american}.
\end{enumerate}
\item In the boundary case $r = \frac12 b^2$, the convexity of $H$ is determined by the sign of the difference $D - p_c$. As above the possibility $D > p_c$ is excluded since then $H$ is strictly convex. Otherwise $H$ is concave and the stopping sets $\Gamma$ have the same form as in case 1(ii) above.
\end{enumerate}

\subsubsection{OU imbalance process}\label{sec:oui}

When $Z$ is the Ornstein-Uhlenbeck process, by adjusting $d$ and $b$ in the price stack function $f$ (see \eqref{eq:fz}) we can restrict our analysis to the OU process with zero mean and unit volatility, that is:
\[
dZ_t = -\theta Z_t dt + dW_t.
\]
Then for $z<z^*$
\begin{eqnarray}
(\mathcal{L}-r)\hhf(z,\hz) &=& (\mathcal{L}-r)(-f(z) + p_c) \\
&=& de^{bz}\left\{b\left(\theta z - \frac12 b \right) + r \right\} + r(D - p_c) =: \eta(z).\label{eqn:OUstack}
\end{eqnarray}
Differentiating $\eta$ we obtain
\[
\eta'(z) = db\theta e^{bz} \left( bz + 1 + \frac{r - \frac12 b^2}{\theta} \right)
\]
which has a unique root at $z^\diamond = \frac{1}{b} \Big( \frac{\frac12 b^2 - r}{\theta} - 1 \Big)$. The function $\eta$ decreases from $r(D-p_c)$ at $-\infty$ until $\eta(z^\diamond) = -de^{bz^\diamond} \theta + r(D-p_c)$ at  $z^\diamond$ and then increases to positive infinity.

\begin{enumerate}
\item If $p_c \ge D$ then the function $\eta$ is negative on $(-\infty,u)$, where $u$ is the unique root of $\eta$. Hence $H$ is concave on $(0, F^Z(u))$ and convex on $(F^Z(u), \infty)$. The stopping sets $\Gamma$ for the single and lifetime problems must then be of the form $(-\infty, \hat z]$ and $(-\infty, \tilde z]$, respectively, c.f. case 1(ii) in Subsection \ref{sec:bmi}.
\item The case $p_c < D$ is more complex. 
\begin{itemize}
\item [(i)] Let $z^\diamond \ge z^*$. We exclude the possibility $\eta(z^*) \ge 0$, since then the function $H$ is convex on $(0, F^Z(z^*))$ and the set $\Gamma$  has empty intersection with this interval, contradicting Theorem \ref{thm:hammerA'} and, consequently, violating \textbf{S1$^*$} or \textbf{S2$^*$}. When $\eta(z^*) < 0$, $H$ is convex on $(0, F^Z(u))$ and concave on $(F^Z(u), F^Z(z^*))$, where $u$ is the unique root of $\eta$ on $(0, z^*)$. Therefore the stopping sets $\Gamma$ for the single and lifetime problems are of the form $[\hat z_0, \hat z]$ and $[\tilde z_0, \tilde z]$, respectively, with $\min(\hat z_0,\tilde z_0) > u$, c.f. case 2(ii) in Subsection \ref{sec:bmi}.
\item[(ii)] Consider now $z^\diamond < z^*$. As above we exclude the case $\eta(z^\diamond) \ge 0$, since then $H$ is convex on $(0, F^Z(z^*))$. The remaining case $\eta(z^\diamond) < 0$ implies that the stopping sets $\Gamma$ have the same form as in case 2(i) above, as $H$ is convex and then concave if $\eta(z^*) \le 0$, and convex-concave-convex if $\eta(z^*)> 0$.
\end{itemize}
\end{enumerate}

\vspace{1mm}

\section{Empirical illustration and qualitative implications}\label{sec:empirics}

In this section we present an empirical illustration of the above results and draw qualitative implications from our work. The strength of the present paper is in providing a framework capable of accommodating advanced EIM price models through Theorem \ref{thm:hammerA}, and, for shifted exponential price models, through Theorem \ref{thm:hammerA'}. Nevertheless the detailed modelling of EIM prices is beyond the scope of this paper and we assume the OU model of Section \ref{subsec:OU},  which captures both the mean reversion and random variability present in the EIM prices, fitting this model to relevant data in Section \ref{sec:oufit}. We take an interest rate of $3\%$ per annum and the degradation factor for the store to be $A=0.9999$.

\subsection{Imbalance market data and fitted OU model}\label{sec:oufit}
Our data is the `balancing group price' from the German Amprion SO, which is available for every 15 minute period \citep{amprion2016}. Summary statistics for the average daily price from 1 June 2012 to 31 May 2016 are presented in Table \ref{tab:ss1}. To address the issue of its extreme range, which impacts the fitting of both volatility and mean reversion in the OU model, the data was truncated at the values -150 and 150. The parameters obtained by maximum likelihood fitting were then $\theta = 68.69$, $\sigma = 483.33$, $D = 30.99$. (The effect of the truncation step was to approximately halve the fitted volatility.)

\begin{table}[htp]
\caption{Summary statistics for the daily average balancing group price in the German Amprion area,  1 June 2012 to 31 May 2016.
}
\begin{center}
\begin{tabular}{|c|c|c|c|c|c|}
\hline
Min. & 1st Qu. &  Median &    Mean & 3rd Qu. &    Max. \\
\hline
-6002.00  &   0.27 &   33.05 &   31.14  &  66.97 & 6344.00 \\
\hline
\end{tabular}
\label{tab:ss1}
\end{center}
\label{default}
\end{table}

\subsection{Illustration and implications}

\begin{figure}[tb]
\begin{center}
\includegraphics[width=0.45\textwidth]{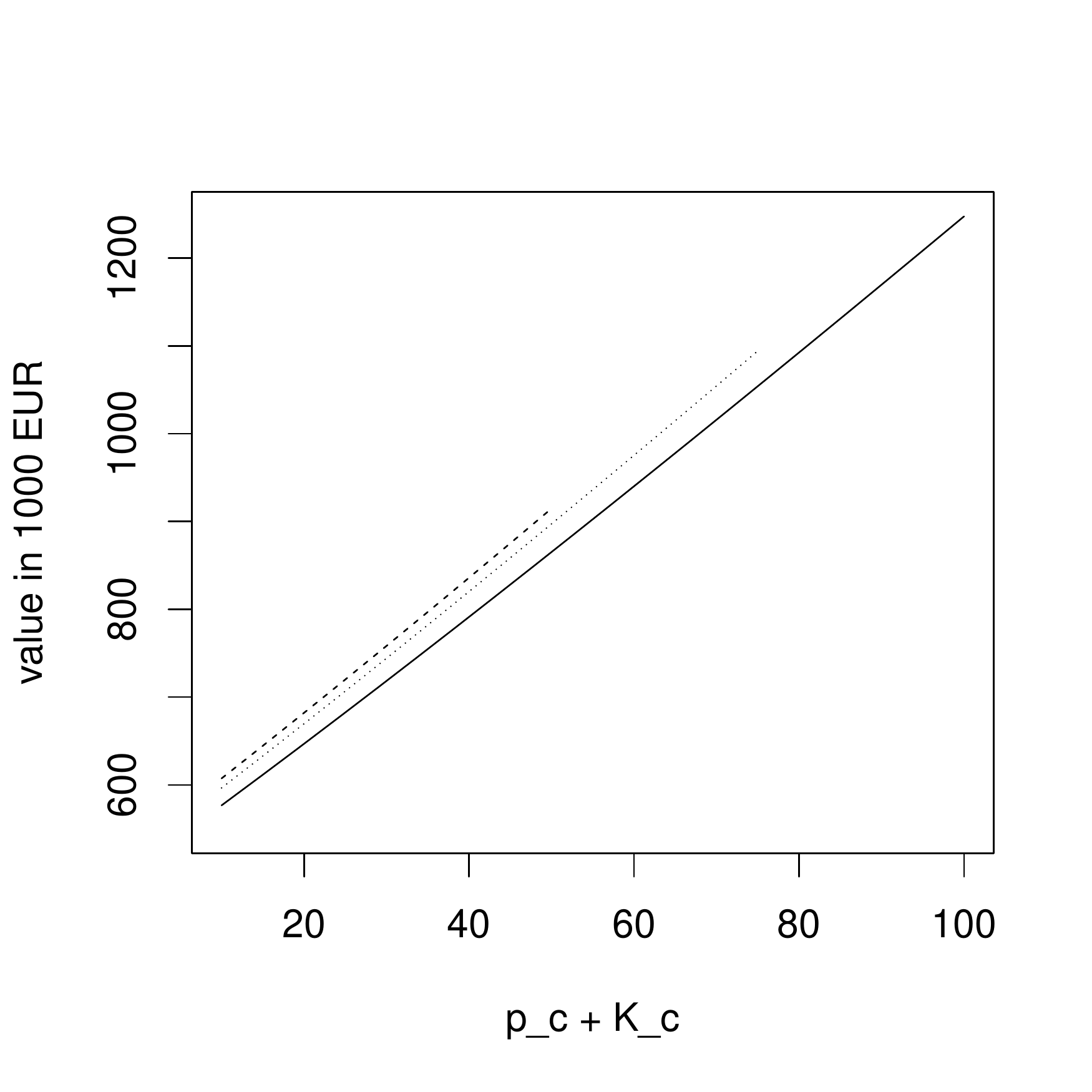}
\qquad
\includegraphics[width=0.45\textwidth]{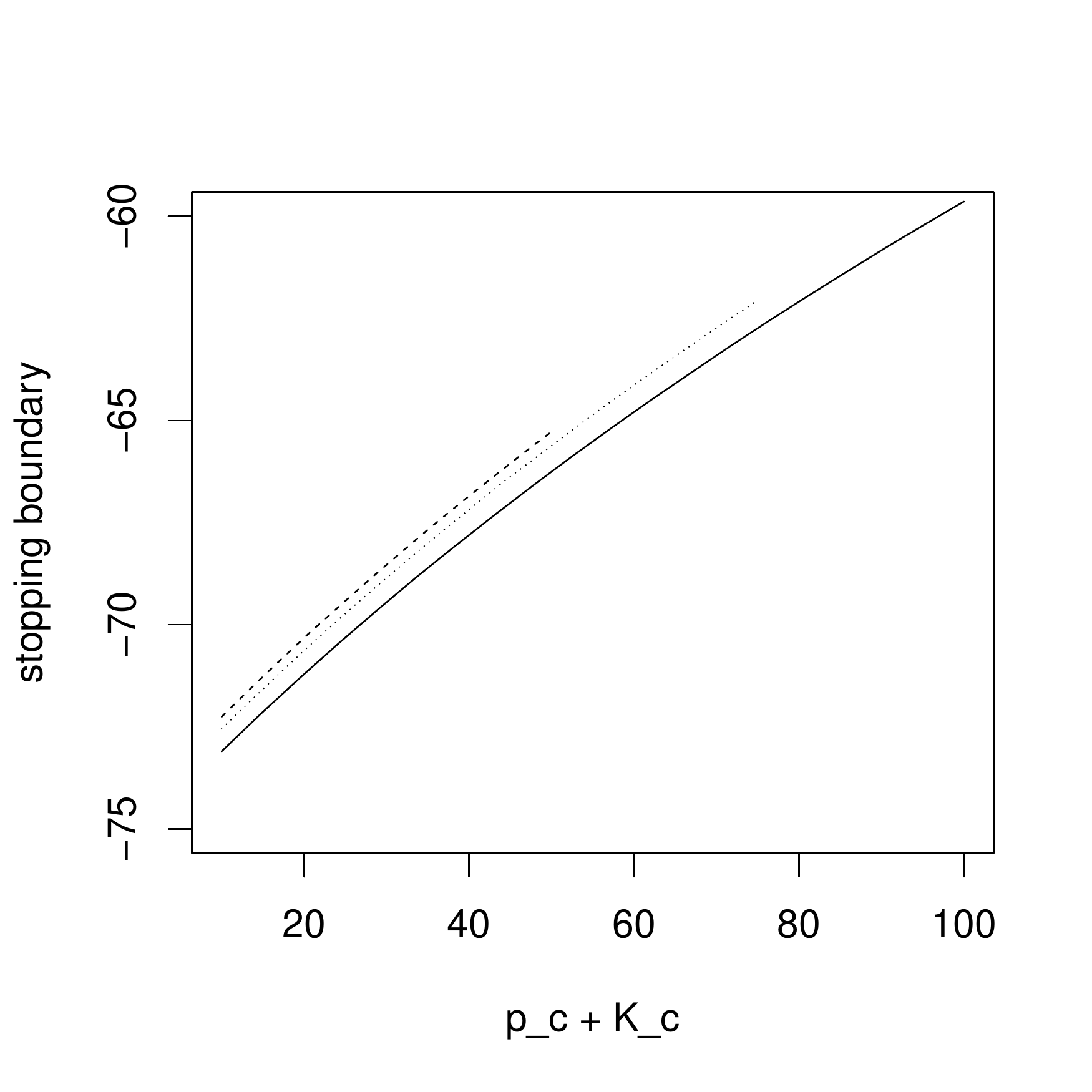}
\caption{Results obtained with the Ornstein-Uhlenbeck model fitted in Section \ref{sec:oufit}, as functions of the total premium, with interest rate $3\%$ per annum. Solid lines: $x^* = 100$, dotted: $x^* = 75$, dashed: $x^*=50$. Left: lifetime value $\hat V(x^*)$. Right: the stopping boundary $\check x$, the maximum price for which BO can buy energy optimally. }
\label{fig:singleplot}
\end{center}
\end{figure}

\begin{figure}[tb]
\begin{center}
\includegraphics[width=0.5\textwidth]{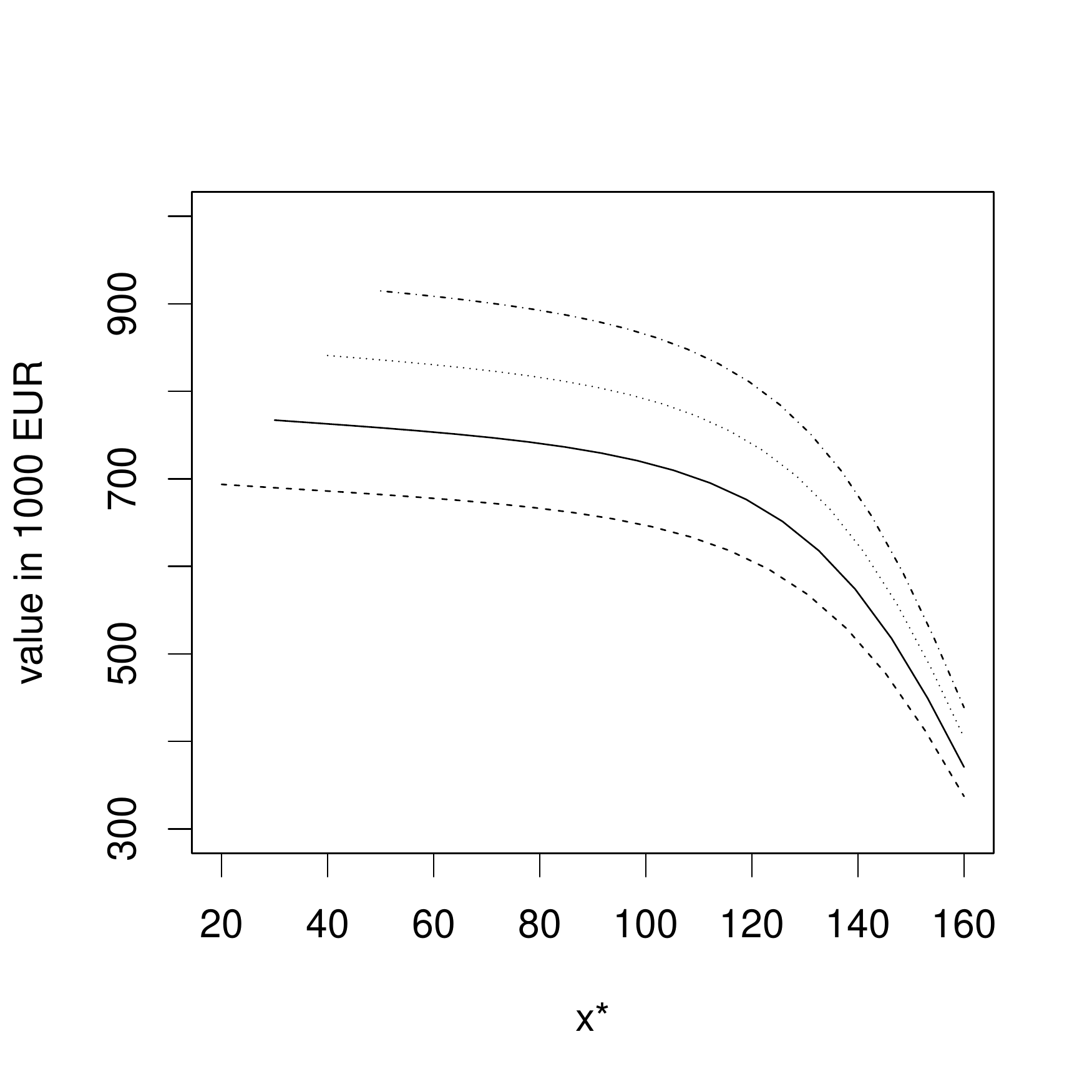}
\end{center}
\caption{Lifetime value $\hat V(x^*)$ as a function of $x^*$ with the Ornstein-Uhlenbeck model fitted in Section \ref{sec:oufit}, with interest rate $3\%$ per annum. Dashed line: $p_c + K_c = 20$, solid: $p_c + K_c = 30$, dotted: $p_c + K_c = 40$, mixed: $p_c + K_c = 50$.}
\label{fig:value_premium}
\end{figure}

The left panel of Figure \ref{fig:singleplot}, and Figure \ref{fig:value_premium}, show
 the lifetime value $\hat V(x^*)$, while the right panel of Figure \ref{fig:singleplot} plots the stopping boundary $\check x$, which is the maximum price at which the BO can buy energy optimally. These values of $\check x$ are significantly below the long-term mean price $D$, indeed the former value is negative while the latter is positive. Thus in this example the BO purchases energy when it is in excess supply, further contributing to balancing. To place the negative values on the stopping boundary in Figure \ref{fig:singleplot} in the statistical context, recall from Table \ref{tab:ss1} that the first quartile of the price distribution is approximately zero. Indeed negative energy prices usually occur several times per day in the German EIM market. In the present dataset of $1461$ days there are only $11$ days without negative prices and the longest observed time between negative prices is $41.5$ hours.

We make two empirical observations. Firstly, defining the total premium as the sum $p_c + K_c$, altering its distribution between the initial premium $p_c$ (which is received at $x=\check x$) and the utilisation payment $K_c$ (which is received at $x=x^*$) results in insignificant changes to the graphs, with relative differences on the vertical axes of the order $10^{-3}$ (data not shown). It is for this reason that the figures are indexed by the total premium $p_c+K_c$ rather than by individual premia. Secondly, the contours in Figure \ref{fig:value_premium} have a `hockey stick' shape, the marginal influence of $x^*$ being smaller in the range $x^*<110$ and larger for greater values of $x^*$.

These two phenomena are explained by the presence of mean reversion in the OU price model. The timings of the cashflows to the BO are entirely determined by the successive {\em passage times} of the price process between the levels $x^*$ and $\check x$. These passage times are relatively short on average for the fitted OU model. This means that the premia are received at almost the same time under each option contract, and it is the total premium which drives the real option value. Further the passage times between $x^*$ and $\check x$ may be decomposed into passage times between $x^*$ and $D$, and between $D$ and $\check x$. Since the OU process is statistically symmetric about $D$, let us  compare the distances $|\check x-D|$ and $|x^*-D|$. From Figure \ref{fig:singleplot} we have $\check x \approx -70$ so that $|\check x-D| \approx 100$. Therefore for $x^*<110$ we have $|x^*-D| \ll |\check x-D|$ and the passage time between $D$ and $\check x$, which varies little, dominates that between $x^*$ and $D$. Correspondingly we observe in  Figure \ref{fig:value_premium} that the value function changes relatively little as $x^*$ varies below $110$. Conversely, as $x^*$ increases beyond $110$ it is the distance between $x^*$ and $D$ which dominates, and the value function begins to decrease relatively rapidly. 

These results provide insights into the suitability of the present contract structure for correcting differing levels of imbalance. As the distance between $x^*$ and the mean level $D$ grows, the energy price reaches $x^*$ significantly less frequently and the option starts to provide insurance against rare events, resulting in infrequent option exercise and low power flow through the battery. These observations suggest that the American option type contract studied in this paper is more suitable for the frequent balancing of less severe imbalance. In contrast, 
the more rapid reduction in the lifetime value for large values of $x^*$ suggests that these options and, more generally, market based arrangements are not suitable for balancing relatively rare events such as large system disturbances due to unplanned outages of large generators. The SO may prefer to use alternative arrangements, based for example on fixed availability payments, to provide security against such events.

\section{Summary}\label{sec:conclusion}

In this paper we investigate the procurement of operating reserve from energy-limited storage using a sequence of physically covered American style call options. In order to perform real options analysis we take a general linear regular diffusion model of an energy imbalance market price.  In particular our methodology is capable of modelling the mean reversion present in imbalance prices, and we have also taken account of multiplicative degradation in the capacity of the store. Both the optimal operational policy and the real option value of the store are characterised explicitly, for both a single option and the lifetime problem (an infinite sequence of options traded back-to-back). Although the solutions are generally not available in an analytical form we have provided a straightforward procedure for their numerical evaluation together with empirical examples from the German energy imbalance market.

The results of the lifetime analysis in particular have both managerial implications for the BO and policy implications for the SO. From the operational viewpoint, under the setup described in Section \ref{sec:setup} we have established that the BO should purchase energy as soon as the EIM price falls to the level $\check x$, which may be calculated as described in Section \ref{sec:lifetime_constr}. Further the BO should then sell the call option immediately. Our real options valuation may be taken into account when deciding whether to invest in an energy store, and whether to offer such options in preference to trading in other markets (for example, performing price arbitrage in the spot generation market).

Turning to the perspective of the SO, we have demonstrated an American option type contract structure with physical cover, which may be seen as preferable to auctions or long-term bilateral contracts for procuring balancing reserve from energy-limited resources such as batteries. 
Our analysis shows that this contractual arrangement can be mutually beneficial to the SO and BO. More precisely, the SO can be protected against guaranteed financial losses
from the option purchase while the BO has a quantifiable profit. The analysis also provides information on feedback due to battery charging by determining  the highest price $\check x$ at which the BO buys energy, hence identifying conditions under which the BO's operational strategy is aligned with system stability. We have argued that the American option style contract is more suited to the frequent balancing of less severe imbalance, because in such conditions it is exercised sufficiently often to justify the use of a two-part (availability and utilisation) payment structure.
For extreme events a contract with a continuous availability payment appears to be more suitable and such frameworks are already in use by system operators.

We address call options, which are particularly valuable to the SO when the margin of electricity generation capacity over peak demand is low. Put options may also be studied in the above framework, although in the put case the second stopping time (action A2) is non-trivial which leads to a nested stopping problem beyond the scope of the present paper. Further we assume that the energy storage unit is dedicated to providing EIM call options, so that the opportunity costs of not operating in other markets or providing other services are not modelled. The extension to a finite expiry time, the lifetime analysis of the put option, and also the opportunity cost of not operating in other markets would be interesting areas for further work.

The methodological advances of this paper reach beyond energy markets. In particular they are relevant to real options analyses of storable commodities where the timing problem over the lifetime of the store is of primary interest. The novel lifetime analysis via optimal stopping techniques, developed in Section \ref{sec:mainres}, provides an example of how timing problems can be addressed for rather general dynamics of the underlying stochastic process. In this context we provide an alternative method to quasi-variational inequalities, which are often dynamics-specific and technically more involved.

\medskip

{\bf Acknowledgement.} JM gratefully acknowledges support by grant EP/K00557X/1 from the UK Engineering and Physical Sciences Research Council. JP was supported in part by MNiSzW grant UMO-2012/07/B/ST1/03298.

\bibliographystyle{elsarticle-harv.bst}
\bibliography{OR}

\begin{appendices}

\setcounter{table}{0}
\setcounter{figure}{0}
\setcounter{equation}{0}
\counterwithin*{figure}{section}
\counterwithin*{table}{section}
\counterwithin*{equation}{section}
\renewcommand{\thetable}{{\thesection}.\arabic{table}}
\renewcommand{\thefigure}{{\thesection}.\arabic{figure}}
\renewcommand{\theequation}{{\thesection}.\arabic{equation}}

\section{Uniqueness of fixed points}
\label{sec:uniqueness}

The following lemmas establish the uniqueness of the fixed point of $\hat \ct$ and the exponential speed of convergence of $\hat\ct^n \mathbf{0}$ to $\hat \zeta$. 

\begin{Lemma} \label{lem:contraction}
Let $\xi, \xi'$ be two continuous non-negative functions with $\xi$ satisfying the assumptions of Lemma \ref{lem:char} together with the bound $\xi \ge \xi'$. In the problem $\hat \ct \xi$, assume the existence of an optimal stopping time $\tau^*$ under which stopping occurs only at values bounded above by $x' < x^*$. Then
\[
\| \hat \ct \xi - \hat \ct \xi' \|_\# \le \rho \| \xi - \xi' \|_\#,
\]
where $\rho = A \frac{\psi_r(x')}{\psi_r(x^*)} < 1$ and $\| f \|_\# = |f(x^*)|$ is a seminorm on the space of continuous functions. Moreover,
\begin{equation}\label{eqn:ct_seminorm_bound}
0 \le \hat \ct \xi(x) - \hat \ct \xi' (x) < \| \xi - \xi' \|_\#.
\end{equation}
\end{Lemma}

Note that in general, an optimal stopping time for $\hat\ct \xi(x)$ depends on the initial state $x$. However, under general conditions (cf. Section \ref{sec:osmethod}), $\tau^* = \inf\{t \ge 0:\ X_t \in \Gamma \}$, where $\Gamma$ is the stopping set. Then the condition in the above lemma writes as $\Gamma \subset (a, x']$ for some $x' < x^*$.
\begin{proof}{Proof of Lemma \ref{lem:contraction}}
By the monotonicity of $\hat \ct$, for any $x$ we have
\begin{align*}
0 \le \hat \ct \xi(x) - \hat \ct \xi' (x)
& \le \ee^x \Big\{e^{-r\tau^*} \Big(-X_{\tau^*} + p_c + \big( K_c + A\xi(x^*) \big) \frac{\psi_r(X_{\tau^*})}{\psi_r(x^*)}\Big) \Big\}\\
&\hspace{11pt} - \ee^x \Big\{e^{-r\tau^*} \Big(-X_{\tau^*} + p_c + \big( K_c + A\xi'(x^*) \big) \frac{\psi_r(X_{\tau^*})}{\psi_r(x^*)}\Big) \Big\}\\
& = \ee^x \Big\{e^{-r\tau^*} A \Big( \big(\xi(x^*) - \xi'(x^*) \big) \frac{\psi_r(X_{\tau^*})}{\psi_r(x^*)}\Big) \Big\}\\
&= \| \xi - \xi' \|_\# A \, \ee^x \Big\{e^{-r\tau^*} \frac{\psi_r(X_{\tau^*})}{\psi_r(x^*)} \Big\}.
\end{align*}
This proves \eqref{eqn:ct_seminorm_bound}. Also we have
\[
A\, \ee^{x^*} \Big\{e^{-r\tau^*} \frac{\psi_r(X_{\tau^*})}{\psi_r(x^*)} \Big\}
\le A \frac{\phi_r(x^*)}{\phi_r(x')} \frac{\psi_r(x')}{\psi_r(x^*)}
\le \rho.
\]
\end{proof}

\begin{Lemma} \label{lem:geometric_convergence}
Assume that there exists a fixed point $\hat\zeta^*$ of $\hat \ct$ in the space of continuous non-negative functions. In the problem $\hat \ct \hz^*$, assume the existence of an optimal stopping time under which stopping occurs only at values bounded above by $x' < x^*$ (c.f. the comment after the previous lemma).
Then there is a constant $\rho < 1$ such that \( \| \hat\zeta^* - \hat \ct^n \mathbf{0}\|_\# \le \rho^n \| \hat\zeta^* \|_\# \)
and $\|\hat\zeta^* - \hat \ct^n \mathbf{0}\|_\infty \le \rho^{n-1} \| \hat\zeta^*\|_\#$, where $\|\cdot\|_\infty$ is the supremum norm.
\end{Lemma}
\begin{proof}{Proof.}
Clearly, $\|\hat\zeta^* - \mathbf{0} \|_\# < \infty$. By virtue of Lemma \ref{lem:contraction} we have $\|\hat \ct^n \mathbf{0} - \hat\zeta^*\|_\# \le \rho^n \|\mathbf{0} - \hat\zeta^*\|_\#$ for $\rho =  \frac{\psi_r(x')}{\psi_r(x^*)} < 1$. Hence, $\hat \ct^n \mathbf{0}$ converges exponentially fast to $\hat\zeta^*$ in the seminorm $\|\cdot\|_\#$. Using \eqref{eqn:ct_seminorm_bound} we have
\[
\| \hat\zeta^* - \hat \ct^n \mathbf{0} \|_\infty = \| \hat \ct \hat\zeta^* - \hat \ct \circ \hat \ct^{n-1} \mathbf{0} \|_\infty \le \rho^{n-1} \| \hat\zeta^*\|_\#.
\]
\end{proof}

\section{Note on Lemma \ref{cor:suff_case_A}}
\label{sec:lemrmk}

The inequality $\lim_{x \to -\infty} \frac{-x}{\phi_r(x)} > 0$ when $a=-\infty$ asserts that the process $X$ escapes to $-\infty$ quickly. Indeed, choosing $z \in I$, we have $\ee^z \{ e^{-r\tau_x} \} = \frac{\phi_r(z)}{\phi_r(x)}$ for $x \le z$, hence $\ee^z \{ e^{-r\tau_{x}} \} \ge \frac{c}{-x}$ for some constant $c > 0$ and $x$ sufficiently close to $-\infty$. To illustrate the speed of escape, assume for simplicity that $X$ is a deterministic process.
Then the last inequality would imply $\tau_x \le \frac{1}{r} \big(\log(-x) - \log(c)\big)$, i.e., $X$ escapes to $-\infty$ exponentially quickly. 

An example of a model that violates the assumptions of Lemma \ref{cor:suff_case_A} is the negative geometric Brownian motion: $X_t = - \exp \big((\mu - \sigma^2/2) t + \sigma W_t\big)$ for $\mu, \sigma > 0$. With the generator $\mathcal{A} = \frac{1}{2}\sigma^2 x^2 \frac{d^2}{dx^2} + \mu x \frac{d}{dx}$, we have $\phi_r(x) = (-x)^{\gamma_2}$ and $\psi_r(x) = (-x)^{\gamma_1}$, where $\gamma_1 < 0 < \gamma_2$ are solutions to the quadratic equation $\frac{\sigma^2}{2} \gamma^2 + (\mu - \frac{\sigma^2}{2}) \gamma - r = 0$, i.e., $
\gamma = B \pm \sqrt{B^2 + 2 \frac{r}{\sigma^2}}$ with $B = \frac12 - \frac{\mu}{\sigma^2}$. Hence, $\lim_{x \to -\infty} \frac{-x}{\phi_r(x)}  = \lim_{x \to -\infty} (-x)^{1-\gamma_2} > 0$ if and only if $\gamma_2 \le 1$. It is easy to check that $\gamma_2 = 1$ for $\mu = r$ and $\gamma_2$ is decreasing as a function of $\mu$. Therefore, the condition $\gamma_2 \le 1$ is equivalent to $\mu \ge r$.

In summary, the negative geometric Brownian motion violates the assumptions of Lemma \ref{cor:suff_case_A} if $\mu \ge r$. If $\mu=r$ then case $B$ of Theorem \ref{thm:hammerA} applies with $L_c=1$, while if $\mu > r$ then $L_c = \infty$ and so case C applies. Both cases may be interpreted heuristically as the negative geometric Brownian motion $X$ escaping `relatively quickly' to $-\infty$, that is, relative to the value $r$ of the continuously compounded interest rate. In the latter case this happens sufficiently quickly that the single option value function $V_c$ is infinite.

\section{Facts about the OU process}\label{sec:oufacts}

Let us temporarily fix $\mu=0$ and $\theta=\sigma=1$. Consider the ordinary differential equation (ODE)
\[
w''(z) + \left(\nu + \frac12 - \frac14 z^2\right) w(z) = 0.
\]
There are two fundamental solutions $D_{\nu} (z)$ and $D_\nu (-z)$, where $D_\nu$ is a parabolic cylinder function. Assume that $\nu < 0$. This function has a multitude of representations, but the following will be sufficient for our purposes \citep[p. 119]{erdelyi1953higher}:
\[
D_\nu(z) = \frac{e^{-z^2/4}}{\Gamma(-\nu)} \int_0^\infty e^{-zt - \frac12 t^2} t^{-\nu-1} dt.
\]
Then $D_\nu$ is strictly positive. 
Fix $r > 0$. Define
\[
\psi_r(x) = e^{\frac {(x-\mu)^2 \theta}{2\sigma^2}} D_{-r/\theta} \Big(-\frac{(x - \mu)\sqrt{2 \theta}}{\sigma}\Big), \qquad \phi_r(x) = e^{\frac {(x - \mu)^2 \theta}{2\sigma^2}} D_{-r/\theta} \Big(\frac{(x - \mu) \sqrt{2 \theta}}{\sigma}\Big).
\]
By direct calculation one verifies that these functions solve
\begin{equation}\label{eqn:OU_fund_eqn}
\mathcal{L} v = r v,
\end{equation}
where
\begin{equation}\label{eqn:OU_L}
 \mathcal{L} v(x) = \frac12 \sigma^2 v''(x)  + \theta (\mu - x) v'(x)
\end{equation}
is the infinitesimal generator of the OU process \eqref{eq:ousde}. Setting $\nu = -r/\theta$ we can write
\[
\psi_r(x) = \frac1{\Gamma(-\nu)} \int_0^\infty e^{(x-\mu)t \frac{\sqrt{2\theta}}{\sigma} - \frac12 t^2} t^{-\nu-1} dt,
\qquad
\phi_r(x) = \frac1{\Gamma(-\nu)} \int_0^\infty e^{-(x-\mu)t \frac{\sqrt{2\theta}}{\sigma} - \frac12 t^2} t^{-\nu-1} dt.
\]
Hence $\psi_r$ is increasing and $\phi_r$ is decreasing in $x$. Also, by monotone convergence $\psi_r(-\infty) = \phi_r(\infty) = 0$ and $\psi_r(\infty) = \phi_r(-\infty) = \infty$. The functions $\psi_r$ and $\phi_r$ are then fundamental solutions of the equation \eqref{eqn:OU_fund_eqn}. Further they  are strictly convex, which can be checked by passing differentiation under the integral sign (justified by the dominated convergence theorem). 
Defining $F(x) = \psi_r(x) / \phi_r(x)$, then $F$ is continuous and strictly increasing with $F(-\infty) = 0$ and $F(\infty) = \infty$.

Using the integral representation of $\phi_r$ and l'H\^opital's rule we have
\begin{equation}\label{eqn:dx_phi_limit}
\begin{aligned}
\lim_{x \to -\infty} \frac{-x}{\phi_r(x)} &= \lim_{x \to -\infty} \frac{-1}{\frac1{\Gamma(-\nu)} \int_0^\infty e^{-(x-\mu)t \frac{\sqrt{2\theta}}{\sigma} - \frac12 t^2} \Big(-t \frac{\sqrt{2\theta}}{\sigma} \Big)t^{-\nu-1} dt}\\
&= \frac{\sigma}{\sqrt{2\theta}} \lim_{x \to -\infty} \frac{1}{\frac1{\Gamma(-\nu)} \int_0^\infty e^{-(x-\mu)t \frac{\sqrt{2\theta}}{\sigma} - \frac12 t^2} t^{-\nu} dt}\\
&= \frac{\sigma}{\sqrt{2\theta}} \lim_{x \to -\infty} \frac{1}{\frac{\Gamma(-\nu+1)}{\Gamma(-\nu)} \frac{1}{\Gamma(-\nu+1)} \int_0^\infty e^{-(x-\mu)t \frac{\sqrt{2\theta}}{\sigma} - \frac12 t^2} t^{-\nu} dt}
= 0,
\end{aligned}
\end{equation}
as the denominator is a scaled version of $\phi_{\tilde r}$ corresponding to a new $\tilde r$ such that $-\tilde r/\theta = \nu-1 < \nu < 0$, and so it converges to infinity when $x \to -\infty$.

\end{appendices}

\end{document}